\documentclass[a4paper]{article}

\usepackage[english]{babel}
\usepackage[T1]{fontenc}
\usepackage{amssymb}
\usepackage{amsthm}
\usepackage{makeidx}
\usepackage{color}
\usepackage{fullpage}
\usepackage{tikz}
\usepackage{tikz-cd}
\usepackage{mathtools}
\usepackage{mathpazo}
\usepackage[all]{xy}
\usepackage{comment}
\usepackage{authblk}
\usepackage{url}
\usepackage{enumitem}
\usepackage[backend=biber,style=numeric,natbib=true]{biblatex}

\usepackage[a4paper,top=3cm,bottom=2cm,left=3cm,right=3cm,marginparwidth=1.75cm]{geometry}

\usepackage{amsmath}
\usepackage{graphicx}
\usepackage[colorinlistoftodos]{todonotes}
\usepackage[colorlinks=true, allcolors=blue]{hyperref}
\usepackage[all]{xy}

\newtheorem{defi}{Definition}[section]
\newtheorem{example}[defi]{Example}
\newtheorem{thm}[defi]{Theorem}
\newtheorem{coro}[defi]{Corollary}
\newtheorem{lema}[defi]{Lemma}

\newtheorem{pro}[defi]{Proposition}
\newtheorem{rmk}[defi]{Remark}

\providecommand{\keywords}[1]{\textbf{\textit{Keywords---}} #1}

\providecommand{\MSC}[1]{\textbf{\textit{MSC---}} #1}

\title{Hopf-Galois module structure of \\ degree $p$ extensions of $p$-adic fields}
\author[1,2,3]{Daniel Gil-Muñoz}
\date{}

\affil[1]{\footnotesize Institut de Matemàtica, Universitat de Barcelona, Gran Via de les Corts Catalanes, 585, 08007 Barcelona, Spain}
\affil[2]{\footnotesize Charles University, Faculty of Mathematics and Physics, Department of Algebra, Sokolovska 83, 18600 Praha 8, Czech Republic}
\affil[3]{\footnotesize Dipartimento di Matematica, Università di Pisa, Largo B. Pontecorvo, 5, 56127 Pisa, Italy}

\begin{document}
\maketitle

\begin{abstract}
Let $p$ be an odd prime number. For a degree $p$ extension of $p$-adic fields $L/K$, we give a complete characterization of the condition for the ring of integers $\mathcal{O}_L$ to be free as a module over its associated order in the unique Hopf-Galois structure on $L/K$.
\end{abstract}

\MSC{11S15, 12F10, 16T05}


\section{Introduction}

The classical Galois module theory seeks to describe the module structure of the ring of integers of a Galois extension of number or $p$-adic fields. The ground ring of such a module is chosen as an object related with the Galois group. The group algebra of the Galois group (i.e, the Galois group algebra) with scalars in the ground ring of integers of the extension appears as a natural choice. Alternatively, we can choose the associated order of the ring of integers, which is defined as the subset of the Galois group algebra with scalars in the ground field (not just the ring of integers) whose action on the top field leaves its ring of integers invariant. This choice is optimal in the sense that it is the only order in the Galois group algebra over which the ring of integers may be possibly a free module.

This situation is generalized by means of Hopf-Galois theory, whose starting point is the notion of a Hopf-Galois structure on a field extension: a pair formed by a Hopf algebra and an action on the top field of the extension in such a way that it mimics the properties of the Galois group algebra. The extensions that admit some Hopf-Galois structure are called Hopf-Galois. If the extension is Galois, the Galois group algebra together with its natural action is a Hopf-Galois structure on the Galois extension, usually known as the classical Galois structure, but it may admit others. Under this terminology, there are non-Galois extensions that are Hopf-Galois.

The notion of associated order in the Galois group algebra of a Galois extension is naturally generalized to any Hopf-Galois structure on a Hopf-Galois extension $L/K$. In this way, it makes sense to study the module structure of the top ring of integers $\mathcal{O}_L$ over the associated order $\mathfrak{A}_H$ in any Hopf-Galois structure $H$ on the extension. The module structure of $\mathcal{O}_L$ is completely determined if $\mathcal{O}_L$ is free as a module over $\mathfrak{A}_H$. Indeed, in that case we have that $\mathcal{O}_L\cong\mathfrak{A}_H$ as $\mathfrak{A}_H$-modules. Whenever $H$ is the only Hopf-Galois structure on $L/K$, we will denote $\mathfrak{A}_{L/K}\equiv\mathfrak{A}_H$.

The problem of characterizing the condition that $\mathcal{O}_L$ is $\mathfrak{A}_H$-free for a Hopf-Galois structure $H$ on an extension $L/K$ of $p$-adic fields is of long-standing interest, among other reasons because there is not uniformity in the behaviour of $\mathcal{O}_L$ as an $\mathfrak{A}_H$-module for different choices of extensions $L/K$. A class of extensions illustrating this phenomenon is that of cyclic degree $p$ extensions of $p$-adic fields. Such an extension $L/K$ admits a unique Hopf-Galois structure. When $L/K$ is unramified, then it is necessarily Galois, and $\mathcal{O}_L$ is free over the associated order $\mathfrak{A}_{L/K}$. Otherwise, $L/K$ is totally ramified. In that case, a complete characterization of the $\mathfrak{A}_{L/K}$-freeness of $\mathcal{O}_L$ is as follows.

\begin{thm}[F. Bertrandias, J.P. Bertrandias, M.J. Ferton]\label{thm:cyclicpfreeness} Let $L/K$ be a totally ramified cyclic degree $p$ extension of $p$-adic fields and let $e$ be the ramification index of $K/\mathbb{Q}_p$. Let $G\coloneqq\mathrm{Gal}(L/K)=\langle\sigma\rangle$, let $t$ be the ramification jump of $L/K$ and let $a$ be the remainder of $t$ mod $p$.
\begin{enumerate}[label=\arabic*)]
    \item\label{thm:cyclicpfreeness1} If $a=0$, then $\mathfrak{A}_{L/K}$ is the maximal $\mathcal{O}_K$-order in $K[G]$ and $\mathcal{O}_L$ is $\mathfrak{A}_{L/K}$-free.
    \item\label{thm:cyclicpfreeness2} If $a\neq0$, then $\mathfrak{A}_{L/K}$ is $\mathcal{O}_K$-free with generators $\{\pi_K^{-n_i}f^i\}_{i=0}^{p-1}$, where $f=\sigma-\mathrm{Id}$, \\ $n_i=\mathrm{min}_{0\leq j\leq p-1-i}(\nu_{i+j}-\nu_j)$ and $\nu_i=\Big\lfloor\frac{a+it}{p}\Big\rfloor$ for every $0\leq i\leq p-1$.
    \item\label{thm:cyclicpfreeness3} If $a\mid p-1$, then $\mathcal{O}_L$ is $\mathfrak{A}_{L/K}$-free. Moreover, if $t<\frac{pe}{p-1}-1$, the converse holds.
    \item\label{thm:cyclicpfreeness4} If $t\geq\frac{pe}{p-1}-1$, then $\mathcal{O}_L$ is $\mathfrak{A}_{L/K}$-free if and only if the length of the expansion of $\frac{t}{p}$ as continued fraction is at most $4$.
\end{enumerate}
\end{thm}

It is known that $0<t\leq\frac{pe}{p-1}$, so these statements cover all the possibilities. All of them are stated in \cite{bertrandiasbertrandiasferton,bertrandiasferton}, and in some cases short sketches of proof are given. Complete and detailed proofs appear in Ferton's PhD dissertation \cite[Chapitre II]{fertonthesis}. The first three statements have been reworked by Del Corso, Ferri and Lombardo \cite{delcorsoferrilombardo}, who provided an alternative proof using the notion of minimal index of a Galois extension of $p$-adic fields. Moreover, Byott, Childs and Elder \cite{byottchildselder} used their theory of scaffolds to prove the characterization of the freeness provided by Theorem \ref{thm:cyclicpfreeness} \ref{thm:cyclicpfreeness3} under the slightly weaker assumption that $t<\frac{pe}{p-1}-2$ (see Example 3.3 in the aforementioned reference). A scaffold on a degree $p$ extension of $p$-adic fields $L/K$ is an element $\Phi$ from a $K$-algebra $A$ acting $K$-linearly on $L$ together with elements $\{\lambda_n\}_{n\in\mathbb{Z}}$ of $L$ in such a way that the $L$-valuations $v_L(\Phi\cdot\lambda_n)$ are determined up to a precision $\mathfrak{c}\in\mathbb{Z}_{\geq1}$. Under the assumption that the precision has specific lower bounds, criteria for the freeness of $\mathcal{O}_L$ can be derived.

In this paper, we will consider the problem of studying the module structure of the ring of integers in an arbitrary degree $p$ extension $L/K$ of $p$-adic fields. Of course, such an extension $L/K$ is not necessarily Galois, but it is always Hopf-Galois, and the Galois group $G$ of its normal closure $\widetilde{L}$ over $K$ satisfies $G\cong C_p\rtimes C_r$, where $r\mid p-1$. Consequently, $G$ admits a unique order $r$ subgroup, whence $\widetilde{L}/K$ possesses a intermediate field $M$ such that $[M:K]=r$. It was found by Childs \cite{childs1989} that $L/K$ admits a unique Hopf-Galois structure $H$. Therefore, the relevant problem in this context is to characterize the condition that $\mathcal{O}_L$ is $\mathfrak{A}_{L/K}$-free. The case that $L/K$ is Galois corresponds to the choice $r=1$, for which the solution is given in Theorem \ref{thm:cyclicpfreeness}. The case that $r=2$ is the one for which $G$ is a dihedral group of order $2p$. For these extensions, the author found a complete characterization, which is exactly the same as in the Galois case when $\widetilde{L}/L$ is unramified \cite[Proposition 4.1]{gildegpdihedral} and has a similar flavor otherwise \cite[Theorem 1.2]{gildegpdihedral}. In this last case,  the relevant number is the remainder mod $p$ of the number $\frac{p+t}{2}$, where $t$ is the largest ramification jump of $\widetilde{L}/K$. From now on, for a Galois extension of $p$-adic fields with Galois group $G$, we ignore any possible jump prior to $G_1$ in the chain of ramification groups, as it is done for instance in \cite[\textsection 1.2]{berge1978}.

We shall provide a complete picture for the $\mathfrak{A}_{L/K}$-freeness of $\mathcal{O}_L$ with no restriction on our degree $p$ extension $L/K$. In Section \ref{sec:redtotram} we will show that there is no loss of generality in assuming that $\widetilde{L}/K$ is totally ramified. Namely, calling $K'$ (resp. $L'$) the inertia field of $\widetilde{L}/K$ (resp. $\widetilde{L}/L$), the extension $L'/K'$ is of degree $p$ with normal closure $\widetilde{L}$, which is totally ramified, and $\mathcal{O}_L$ is $\mathfrak{A}_{L/K}$-free if and only if $\mathcal{O}_{L'}$ is $\mathfrak{A}_{L'/K'}$-free. 

Our main result establishes a characterization for the $\mathfrak{A}_{L/K}$-freeness of $\mathcal{O}_L$ when $\widetilde{L}/K$ is totally ramified. Whenever $L/K$ is not Galois, its ramification jump $\ell$ is defined from its Herbrand function; see \cite[Chapter IV, \textsection 3, Remark 2]{serre}. In this case, we have $\ell=\frac{t}{r}$, where $t$ is the ramification jump of $\widetilde{L}/K$. This quantity has already appeared in \cite[Definition before Theorem 2.7]{greither1992} as the phantom ramification number of the extension. It is an integer number only when $r=1$, but in any case it is a $p$-adic integer. It will play the role of $t$ in Theorem \ref{thm:cyclicpfreeness}. Under this perspective, the relevant quantity turns out to be the residue class mod $p$ of $\ell$, also denoted by $a$. 

\begin{thm}\label{maintheorem} Let $L/K$ be a degree $p$ extension of $p$-adic fields with normal closure $\widetilde{L}$, and suppose that $\widetilde{L}/{K}$ is totally ramified. Let $G\coloneqq\mathrm{Gal}(\widetilde{L}/K)$ and let $\mathfrak{A}_{L/K}$ be the associated order of $\mathcal{O}_L$ in the unique Hopf-Galois structure on $L/K$. Write $e$ for the ramification index of $K/\mathbb{Q}_p$, $\ell$ for the ramification jump of $L/K$ and $a$ for the residue class mod $p$ of $\ell$. In addition, let $c$ the remainder of $r\ell$ mod $r$ and $b=\ell-\frac{pc}{r}$.
\begin{enumerate}[label=\arabic*)]
    \item\label{mainthm1} If $a=0$, then $\mathfrak{A}_{L/K}$ is the maximal $\mathcal{O}_K$-order in $H$ and $\mathcal{O}_L$ is $\mathfrak{A}_{L/K}$-free.
    \item\label{mainthm2} If $a\neq0$, $\mathfrak{A}_{L/K}$ is $\mathcal{O}_K$-free with generators $\{\pi_K^{-w(i)}\Psi^i\}_{i=0}^{p-1}$, where $\Psi\in M[\sigma]$ is as in Theorem \ref{thm:hopfalgdimp}, \\ $w(i)=\mathrm{min}_{0\leq j\leq p-1-i}(d(i+j)-d(j))$ and $d(i)=\Big\lfloor\frac{a+ib}{p}\Big\rfloor$ for every $0\leq i\leq p-1$.
    \item\label{mainthm3} If $a\mid p-1$, then $\mathcal{O}_L$ is $\mathfrak{A}_{L/K}$-free. Moreover, if $\ell<\frac{pe}{p-1}-1$, then the converse holds.
    \item\label{mainthm4} If $\ell\geq\frac{pe}{p-1}-1$, then $\mathcal{O}_L$ is $\mathfrak{A}_{L/K}$-free if and only if the length of the continued fraction expansion of $\frac{b}{p}$ is at most $4$.
\end{enumerate}
\end{thm}

Since $0<\ell\leq\frac{pe}{p-1}$, the statements \ref{mainthm1}-\ref{mainthm4} cover all the possibilities for a totally ramified degree $p$ extension of $p$-adic fields. Moreover, if we specify the extension $L/K$ to be Galois, we recover Theorem \ref{thm:cyclicpfreeness} (with the exception that $\Psi\neq f$, so in that case Theorem \ref{maintheorem} provides an alternative basis).

The techniques used in the proofs are substantially different at each case of Theorem \ref{maintheorem}. The case that $a=0$ is considered at Section \ref{sec:maxram} and corresponds to the situation in which $L/K$ is maximally ramified, meaning that $\ell=\frac{pe}{p-1}$ achieves its greatest possible value. The extension $\widetilde{L}/M$ is a Galois degree $p$ extension of $p$-adic fields, which satisfies the conditions of Theorem \ref{thm:cyclicpfreeness} \ref{thm:cyclicpfreeness1}, and some of the arguments used in the proof can be adapted to the extension $L/K$. The maximal ramification at $\widetilde{L}/M$ implies that $\widetilde{L}=M(\gamma)$ for a suitable $p$-th root $\gamma$ of $\pi_M$, which is an eigenvector of all the elements of $\mathrm{Gal}(\widetilde{L}/M)$ (regarded as $M$-endomorphisms of $\widetilde{L}$). From this situation, we shall find a $p$-th root $\alpha$ of $\pi_K$ for which $L=K(\alpha)$ and which is an eigenvector of all the elements of $H$ (regarded as $K$-endomorphisms of $L$). From this description of the action of $H$ one can study the module structure of $\mathcal{O}_L$ so as to prove Theorem \ref{maintheorem} \ref{mainthm1}.

The arguments used for the remaining cases hinge on a result by Childs \cite{childs1989} consisting on a description of the underlying Hopf algebra at the unique Hopf-Galois structure $H$ on $L/K$, by means of the classification of rank $p$ Hopf algebras by Tate and Oort \cite{tateoort}. In Section \ref{sec:typicalp}, we shall review this description following the formulation at \cite[Chapter 4]{childs}, according to which $H=K[\Psi]$ for a specific element $\Psi$ that is constructed from an eigenspace decomposition of $H^+\coloneqq\mathrm{Ker}(\epsilon_H)$, where $\epsilon_H$ is the counity of $H$ as a $K$-Hopf algebra. Moreover, it follows from this description that $\Psi$ satisfies a degree $p$ equation of the type $\Psi^p=\lambda\Psi$, where $\lambda\in\mathcal{O}_K$. We shall find a description of $\lambda$ from which its $K$-valuation can be explicitly determined (see Corollary \ref{coro:eqdegp}). This fact improves the situation at the already known solutions for the cases $r=1$ and $r=2$, where the chosen generator of $H$ at each case satisfies a much trickier equation (see \cite[Proposition 6.8 (2)]{delcorsoferrilombardo} for the Galois case and \cite[Proposition 3.2]{gildegpdihedral} for the dihedral one), and lightens some of the required calculations towards a proof of the statements \ref{mainthm3} and \ref{mainthm4} at Theorem \ref{maintheorem}. More specifically, the equation $\Psi^p=\lambda\Psi$ allows to calculate easily the products of the basis elements of $\mathfrak{A}_{L/K}$, which will be needed for studying the $\mathfrak{A}_{L/K}$-freeness of $\mathcal{O}_L$.

The degree $p$ extensions with $a\neq0$ have been extensively studied by Elder \cite{elder2018}, who refers to them as \textit{typical}. His main result is a complete description of such extensions with the determination of a primitive element satisfying a specific degree $p$ equation, that depends on the ramification parameters (see \cite[Theorem 2.2]{elder2018}). The ingredient of a Hopf-Galois structure other than its underlying Hopf algebra is the action of the underlying Hopf algebra on the top field of the extension. We shall review Elder's result that $\Psi$ raises valuations of selected elements $\lambda_n\in\mathcal{O}_L$ with $L$-valuation $n\in\mathbb{Z}$ by at least $b$, and exactly $b$ if and only if $p\nmid n$ (see \cite[Theorem 3.5]{elder2018}). In the end, this fact explains the presence of $b$ at the description of $\mathfrak{A}_{L/K}$ from Theorem \ref{maintheorem} \ref{mainthm2}.

These results on the Hopf-Galois structure on $L/K$ will provide a suitable setting to characterize the $\mathfrak{A}_{L/K}$-freeness of $\mathcal{O}_L$. It turns out that $\Psi$ together with the elements $\lambda_n$ form a scaffold on $L/K$ with precision $\mathfrak{c}=pe-(p-1)\ell$. Specifying the general result \cite[Theorem 3.1]{byottchildselder} to obtain criteria from a scaffold, Elder obtained that freeness is equivalent to $a\mid p-1$ under the slightly stronger assumption that $\ell<\frac{pe}{p-1}-2$; see \cite[Corollary 3.6]{elder2018}. In Section \ref{sec:freescaffold}, we shall see that we can actually reach the same conclusions by assuming $\ell<\frac{pe}{p-1}-1$, so as to derive Theorem \ref{maintheorem} \ref{mainthm3}. In analogy with the Galois case, typical degree $p$ extensions in this situation will be referred to as non-almost maximally ramified.

In the same way, the typical degree $p$ extensions with $\ell\geq\frac{pe}{p-1}-1$, will be called almost maximally ramified, and they will be considered in Section \ref{sec:freecontfrac}. In this situation, the necessary and sufficient conditions for freeness from the theory of scaffolds do not apply anymore. We shall prove Theorem \ref{maintheorem} \ref{mainthm4}, according to which the criterion for freeness depends on the length $n$ of the continued fraction expansion of $\frac{b}{p}$. We shall see that the strategy employed in the Galois case (see \cite{bertrandiasbertrandiasferton}, \cite[Chapitre II]{fertonthesis}) is naturally adapted to the general case. This strategy can be summarized as follows. Write $$\frac{b}{p}=[a_0;a_1,\dots,a_n]$$ for the continued fraction expansion of $\frac{b}{p}$. For each $i$ with $1\leq 2i<n$ and each $1\leq r\leq a_{2i}$, let $q_{2i,r}$ be the denominator (in irreducible form) of the even-indexed semiconvergent $[a_0;a_1,\dots,a_{2i-1},r]$. Then the numbers $q_{2i,r}$ parametrize the set $E$ of the integers $1\leq h<p$ such that $1\leq h'<h$ implies $\widehat{h'\frac{a}{p}}>\widehat{h\frac{a}{p}}$ (see Proposition \ref{pro:paramE}), where the hat denotes the fractional part of a rational number. The latter description provides an explicit expression of the numbers $w(i)$ (see Lemma \ref{lem:nuiE}). Given that these numbers appear in the description of an $\mathcal{O}_K$-basis of $\mathfrak{A}_{L/K}$, their values are related with the $\mathfrak{A}_{L/K}$-module structure of $\mathcal{O}_L$. This provides a link between the continued fraction expansion of $\frac{b}{p}$ and the $\mathfrak{A}_{L/K}$-freeness of $\mathcal{O}_L$, which will be exploited so as to obtain a proof of Theorem \ref{maintheorem} \ref{mainthm4}. Our proof remains valid for Galois degree $p$ extensions, but the fact the generator $\Psi$ of the Hopf-Galois satisfies $\Psi^p=\lambda\Psi$, which is much simpler than the one satisfied by the generator $f=\sigma-\mathrm{Id}$ provided in \cite{bertrandiasbertrandiasferton}, we actually obtain an easier proof.



\section{Preliminaries}

Throughout all the paper, all our extensions of fields will be assumed to be finite.

\subsection{Hopf-Galois extensions}

Let $L/K$ be an extension of fields. Suppose that there is a finite-dimensional cocommutative $K$-Hopf algebra $H$ that acts on $L$ by means of a $K$-linear map $\cdot\colon H\otimes_KL\longrightarrow L$. We say that $(H,\cdot)$ is a Hopf-Galois structure on $L/K$ if $L$ is an $H$-module algebra with respect to $\cdot$ (see \cite[Chapter 2, Section 2.1, Page 39]{underwood}) and the map \begin{equation}\label{mapj}
    \begin{array}{rccl}
    j\colon & L\otimes_KH & \longrightarrow & \mathrm{End}_K(L) \\
     & x\otimes h & \longmapsto & y\mapsto x(h\cdot y)
\end{array}
\end{equation} is an isomorphism of $K$-vector spaces. The extension $L/K$ is said to be Hopf-Galois or $H$-Galois if it admits some Hopf-Galois structure. In this context, we shall use the label $H$ for referring both to the underlying Hopf algebra of the Hopf-Galois structure and to the Hopf-Galois structure itself (that is, the pair formed by the underlying Hopf algebra and its action on $L$). If $L/K$ is a Galois extension with Galois group $G$, then $K[G]$ together with its classical action on $L$ is a Hopf-Galois structure on $L/K$, called the classical Galois structure. Moreover, other Hopf-Galois structures may arise in a Galois extension, or in extensions that are not necessarily Galois.

Assume that $L/K$ is separable with normal closure $\widetilde{L}$ and let $G=\mathrm{Gal}(\widetilde{L}/K)$ and $G'=\mathrm{Gal}(\widetilde{L}/L)$. Greither and Pareigis \cite{greitherpareigis} proved that the Hopf-Galois structures on $L/K$ are in bijective correspondence with a special class of permutation subgroups, and each Hopf-Galois structure can be recovered from the corresponding subgroup (see also \cite[Chapter 2]{childs}).

Among the separable Hopf-Galois extensions that are not necessarily Galois, there are those, like the degree $p$ extensions considered in this paper, that are known in the literature as almost classically Galois. These are those separable extensions $L/K$ for which $G'$ admits a normal complement $J$ within $G$, so that $G=J\rtimes G'$. The fixed field $M=\widetilde{L}^J$ will be referred to as the Galois complement of $L/K$. Using Greither-Pareigis theory, it is proved that an almost classically Galois extension is Hopf-Galois and it admits a specific Hopf-Galois structure which, in the case that $J$ is abelian, is given by $$H=M[J]^{G'}=\{h\in M[J]\,\mid\,\tau(h)=h\hbox{ for all }\tau\in G'\},$$ where $G'$ acts on $M$ by means of its classical action on $\widetilde{L}$ and on $J$ by conjugation (which is actually an action because $J$ is a normal subgroup).

Recall that two extensions $L/K$, $M/K$ are linearly disjoint if the compositum $LM$ is canonically isomorphic to the tensor product $L\otimes_KM$ as $K$-algebras. If in addition $H$ is a $K$-Hopf algebra endowing $L$ with $H$-module algebra structure, then $H\otimes_KM$ is an $M$-Hopf algebra endowing $LM$ with $H\otimes_KM$-module algebra structure and $L/K$ is $H$-Galois if and only if $LM/M$ is $H\otimes_KM$-Galois. A particular instance of this phenomenon occurs for an almost classically Galois extension, which is linearly disjoint with its Galois complement, and their compositum is the normal closure.

\subsection{Extensions of $p$-adic fields}\label{sec:prelimpadic}

Let $p$ be an odd prime number. By a $p$-adic field we mean any finite extension of the field $\mathbb{Q}_p$ of $p$-adic numbers. For any $p$-adic field $F$, we call $\mathcal{O}_F$ the ring of integers of $F$, $v_F$ its valuation, $\pi_F$ its uniformizer, $\mathfrak{p}_F$ the ideal generated by $\pi_F$ in $\mathcal{O}_F$ and $\overline{F}$ the residue field of $F$. 

\subsubsection{Ramification theory}

We recall the higher ramification theory for extensions of $p$-adic fields, for which we shall follow \cite[Chapter IV]{serre}. 

Let $E/F$ be a Galois degree $n$ extension of $p$-adic fields and let $G$ be its Galois group. Since $\overline{E}/\overline{F}$ is separable, there is some $x\in\mathcal{O}_E$ such that $\mathcal{O}_E=\mathcal{O}_F[x]$ (see \cite[Chapter III, \textsection6, Proposition 12]{serre}). For an integer number $i\geq-1$, the $i$-th ramification group of $E/F$ is $G_{-1}=G$ if $i=-1$ and $$G_i=\{g\in G\,\mid\,v_E(g(x)-x)\geq i+1\}$$ if $i\geq0$. It is clear that $G_i\supseteq G_{i+1}$ for every $i\geq-1$ and that $G_i$ is trivial for $i$ large enough (namely, $i\geq\mathrm{max}_{g\in G}(g(x)-x)$). Therefore, $\{G_i\}_{i=-1}^{\infty}$ is a filtration of $G$, called the chain of ramification groups of $E/F$. A \textbf{ramification jump} for $E/F$ is an integer $t\geq-1$ such that $G_t\neq G_{t+1}$. The Herbrand function of $E/F$ is the function $\varphi_{E/F}\colon\mathbb{R}_{\geq-1}\longrightarrow\mathbb{R}_{\geq-1}$ defined as $$\varphi_{E/F}(u)=\int_0^u\frac{1}{[G_0:G_t]}dt,$$ where $G_u=G_{\lfloor u\rfloor}$ for all $u\in\mathbb{R}$ (see \cite[Chapter IV, \textsection3]{serre}). It holds that $\varphi_{E/F}$ is continuous and strictly increasing (see \cite[Chapter IV, \textsection3, Proposition 12 a)]{serre}), therefore bijective; call $\psi_{E/F}$ its inverse. If $N$ is an intermediate field of $E/F$, by \cite[Chapter IV, \textsection3, Proposition 15]{serre}, we have transitivity formulas \begin{equation}\label{eq:transitherbrand}
    \varphi_{E/F}=\varphi_{N/F}\varphi_{E/N},\quad \psi_{E/F}=\psi_{E/N}\psi_{N/F}.
\end{equation} 

The notion of Herbrand function extends to the case in which $E/F$ is not Galois as follows: If $E\subseteq\Omega$ and $\Omega/F$ is Galois, the Herbrand function of $E/F$ is defined as $\varphi_{E/F}\coloneqq\varphi_{\Omega/F}\psi_{\Omega/E}$. This definition does not depend on the choice of $\Omega$ because of the equalities \eqref{eq:transitherbrand}. The image of a ramification jump for $\Omega/F$ by $\varphi_{E/F}$ will be referred to as a ramification jump for $E/F$. This might not be an integer number, but in any case it is a $p$-adic integer.



\subsubsection{Module structure of the ring of integers}

Let $L/K$ be an $H$-Galois extension of $p$-adic fields. From \cite[(2.16)]{childs}, we know that $L$ is free of rank one as an $H$-module. Note that if $L/K$ is Galois and $H$ is chosen to be its classical Galois structure, we recover the already known normal basis theorem for Galois extensions. Following this analogy, we will say that any generator $\theta$ of $L$ as an $H$-module generates an $H$-normal basis or that it is an $H$-normal basis generator.

The associated order of $\mathcal{O}_L$ in $H$ is defined as $$\mathfrak{A}_H=\{h\in H\,|\,h\cdot\mathcal{O}_L\subset\mathcal{O}_L\}.$$ This object inherits the underlying ring structure of $H$ and in fact it is an $\mathcal{O}_K$-order in $H$ (it is $\mathcal{O}_K$-free and $K\otimes_{\mathcal{O}_K}\mathfrak{A}_H=H$). Moreover, $\mathcal{O}_L$ is naturally endowed with $\mathfrak{A}_H$-module structure, and it makes sense to ask whether such a module is free. What is more, if $\mathfrak{A}$ is an $\mathcal{O}_K$-order in $H$ such that $\mathcal{O}_L$ is $\mathfrak{A}$-free, then $\mathfrak{A}=\mathfrak{A}_H$. In this sense, the associated order is the best choice in order to study the module structure of $\mathcal{O}_L$.

Given $\theta\in\mathcal{O}_L$ generating an $H$-normal basis, let us call $$\mathfrak{A}_{\theta}=\{h\in H\,|\,h\cdot\theta\in\mathcal{O}_L\}.$$ It is immediate that $\mathfrak{A}_H\subseteq\mathfrak{A}_{\theta}$ and $\mathfrak{A}_{\theta}$ is an $\mathfrak{A}_H$-module. In fact, it is proved that $\mathcal{O}_L=\mathfrak{A}_{\theta}\cdot\theta$ and that $\mathfrak{A}_{\theta}$ is an $\mathfrak{A}_H$-fractional ideal \cite[Proposition 2.3]{gildegpdihedral}. We can study the $\mathfrak{A}_H$-freeness through $\mathfrak{A}_{\theta}$ by the following: 

\begin{pro}\label{pro:hnormalfreeness} Let $L/K$ be a degree $n$ $H$-Galois extension of $p$-adic fields. The following statements are equivalent:
\begin{enumerate}[label=\arabic*)]
    \item\label{pro:hnormalfreeness1} $\mathcal{O}_L$ is $\mathfrak{A}_H$-free.
    \item\label{pro:hnormalfreeness2} There is an element $\theta\in\mathcal{O}_L$ generating an $H$-normal basis for $L$ such that $\mathfrak{A}_H=\mathfrak{A}_{\theta}$.
    \item\label{pro:hnormalfreeness3} There is an element $\theta\in\mathcal{O}_L$ generating an $H$-normal basis for $L$ such that $\mathfrak{A}_{\theta}$ is a ring (namely, a subring of the underlying ring structure of $H$).
    \item\label{pro:hnormalfreeness4} There is an element $\theta\in\mathcal{O}_L$ generating an $H$-normal basis for $L$ such that $\mathfrak{A}_{\theta}$ is principal as $\mathfrak{A}_H$-fractional ideal. 
    \item\label{pro:hnormalfreeness5} For every element $\theta\in\mathcal{O}_L$ generating an $H$-normal basis for $L$, $\mathfrak{A}_{\theta}$ is principal as $\mathfrak{A}_H$-fractional ideal.
\end{enumerate}
\end{pro}

This proposition is proved in \cite{gildegpdihedral}. The equivalence between \ref{pro:hnormalfreeness1}, \ref{pro:hnormalfreeness2}, \ref{pro:hnormalfreeness3} and \ref{pro:hnormalfreeness5} is established in \cite[Proposition 2.4]{gildegpdihedral}. The equivalence between the statements \ref{pro:hnormalfreeness4} and\ref{pro:hnormalfreeness5} is proved in the paragraph immediately following \cite[Proposition 2.4]{gildegpdihedral}. Moreover, it is shown that when the statements hold, any element $\theta$ satisfying at least one of the statements \ref{pro:hnormalfreeness2} or \ref{pro:hnormalfreeness3} actually satisfies both of them.

\subsection{Scaffolds on a degree $p$ extension}\label{sec:scaffoldsp}

The theory of scaffolds was introduced by Byott, Childs and Elder \cite{byottchildselder} as a unification of different works on the module structure of the ring of integers in totally ramified degree $p^n$ extensions of local fields. The utility of this theory is that the existence of a scaffold may imply the validity of some criteria on the module structure of the ring of integers in the extension. Here we outline the main results for the case of totally ramified degree $p$ extensions of $p$-adic fields.

Let $b\in\mathbb{Z}$ be coprime with $p$. Call $\mathbb{S}_p=\{0,1,\dots,p-1\}$ and let $\mathfrak{b}\colon\mathbb{S}_{p}\longrightarrow\mathbb{Z}$ be defined by $\mathfrak{b}(s)=bs$. Since $b$ is coprime with $p$, the map $r\circ(-\mathfrak{b})$ is bijective, where $r\colon\mathbb{Z}\longrightarrow\mathbb{S}_p$ sends each integer number to its remainder mod $p$. Let $\mathfrak{a}\colon\mathbb{S}_p\longrightarrow\mathbb{S}_p$ be its inverse. The definition of a scaffold for our case is as follows (see \cite[Definition 2.3]{byottchildselder} for the general definition):

\begin{defi}\label{def:scaffold} Let $L/K$ be a degree $p$ extension of $p$-adic fields. Let $A$ be a $K$-algebra of dimension $p$ acting $K$-linearly on $L$ and let $\mathfrak{c}\in\mathbb{Z}_{\geq1}$. An $A$-scaffold on $L$ with precision $\mathfrak{c}$ and shift parameter $b$ consists of the following data:
\begin{enumerate}[label=\arabic*)]
    \item\label{def:scaffold1} A family $\{\lambda_n\}_{n\in\mathbb{Z}}$ of elements in $L$ such that $v_L(\lambda_n)=n$ for all $n\in\mathbb{Z}$ and $\lambda_{n_1}\lambda_{n_2}^{-1}\in K$ if and only if $n_1\equiv n_2\,(\mathrm{mod}\,p)$.
    \item\label{def:scaffold2} An element $\Psi\in A$ with $\Psi\cdot1=0$ and the property that for every $n\in\mathbb{Z}$ there is some $u_n\in\mathcal{O}_K^*$ such that:
    $$\begin{cases}
    \Psi\cdot\lambda_n\equiv u_n\lambda_{n+b}\,(\mathrm{mod}\,\lambda_{n+b}\mathfrak{p}_L^{\mathfrak{c}}) &\hbox{ if }\mathfrak{a}(n)\geq1 \\
    \Psi\cdot\lambda_n\equiv0\,(\mathrm{mod}\,\lambda_{n+b}\mathfrak{p}_L^{\mathfrak{c}}) &\hbox{ if }\mathfrak{a}(n)=0 \\
    \end{cases}$$
\end{enumerate}
\end{defi}

\begin{rmk}\normalfont\label{rmk:raisesval} The $L$-valuation of the elements $\Psi\cdot\lambda_n$ can be determined from Definition \ref{def:scaffold} as follows. Suppose that $p\nmid n$. Since $v_L(u_n\lambda_{n+b})<v_L(\lambda_{n+b}\pi_L^{\mathfrak{c}})$, taking valuations at the first congruent in \ref{def:scaffold2} yields $$v_L(\Psi\cdot\lambda_n)=v_L(\lambda_{n+b})=n+b=v_L(\lambda_n)+b.$$ On the other hand, when $p\mid n$, $v_L(\Psi\cdot\lambda_n)\geq n+b+\mathfrak{c}>v_L(\lambda_n)+b$. Thus, we see that $\Psi$ raises valuations of the elements $\lambda_n$ by at least $b$, and exactly this value if and only if $p\nmid n$.
\end{rmk}

Let us denote $\mathfrak{A}_A=\{\gamma\in A\,|\,\gamma\cdot\mathcal{O}_L\subset\mathcal{O}_L\}$. If $A$ together with its action is a Hopf-Galois structure on $L/K$, we recover the usual notion of associated order.

\begin{pro}[\cite{byottchildselder}, Proposition 2.12]\label{pro:anormbasisscaffold} Let $L/K$ be a degree $p$ extension of $p$-adic fields with an $A$-scaffold of precision $\mathfrak{c}\in\mathbb{Z}_{\geq1}$. Let $a$ be an integer number such that $\mathfrak{a}(a)=p-1$. Then any element $\rho\in L$ with $v_L(\rho)=a$ is a free generator of $L$ as an $A$-module. Moreover, $\mathfrak{A}_A$ is an $\mathcal{O}_K$-order in $A$.
\end{pro}

Let $a\in\mathbb{S}_p$ be such that $\mathfrak{a}(a)=p-1$ and denote \begin{equation}\label{eq:paramscaffold}
    \begin{split}
        d(i)&=\Big\lfloor\frac{a+ib}{p}\Big\rfloor,\quad0\leq i\leq p-1,\\w(i)&=\mathrm{min}\{d(i+j)-d(j)\,|\,0\leq j\leq p-1-i\},\quad0\leq i\leq p-1.
    \end{split}
\end{equation}

The result \cite[Theorem 3.1]{byottchildselder} on the module structure of $\mathcal{O}_L$ as an $\mathfrak{A}_A$-module in terms of a scaffold from the $K$-algebra $A$ is specified to our situation as follows.

\begin{thm}\label{thm:criteriascaffold} Let $L/K$ be a degree $p$ extension of $p$-adic fields that admits a scaffold with shift parameter $b$ and precision $\mathfrak{c}$. Let $a\in\mathbb{S}_p$ be such that $\mathfrak{a}(a)=p-1$.
\begin{enumerate}[label=\arabic*)]
    \item\label{thm:criteriascaffold1} If $\mathfrak{c}\geq\mathrm{max}(a,1)$, $\{\pi_K^{-w(i)}\Psi^i\}_{i=0}^{p-1}$ is an $\mathcal{O}_K$-basis of $\mathcal{O}_L$. If in addition $w(i)=d(i)$ for all $0\leq i\leq p-1$, then $\mathcal{O}_L$ is $\mathfrak{A}_A$-free.
    \item\label{thm:criteriascaffold2} If $\mathfrak{c}\geq p+a$, $\mathcal{O}_L$ is $\mathfrak{A}_A$-free if and only if $w(i)=d(i)$ for all $0\leq i\leq p-1$.
\end{enumerate}
\end{thm}

The sufficient condition at \ref{thm:criteriascaffold1} is clearly implied by the one at \ref{thm:criteriascaffold2}. Therefore, the latter is commonly known as the weak condition, while the former is the strong one. The weak condition yields a sufficient condition for the freeness, while the strong one implies that such a condition is also necessary, i.e. a characterization. On the other hand, the weak condition is enough to obtain an $\mathcal{O}_K$-basis of $\mathfrak{A}_{L/K}$. 

\begin{example}\normalfont\label{ex:galoisscaffoldp} Byott, Childs and Elder \cite[Example 2.8]{byottchildselder} constructed a scaffold for a Galois degree $p$ extension $L/K$ of $p$-adic fields, ramified with ramification jump $t$, and $e=e(K/\mathbb{Q}_p)$. The shift parameter is $t$ and the precision is $\mathfrak{c}=pe-(p-1)t$. In this case, the integers $d(i)$, $w(i)$ are actually the numbers $\nu_i$, $n_i$ from Theorem \ref{thm:cyclicpfreeness} \ref{thm:cyclicpfreeness2}. The criterion $w(i)=d(i)$ for all $0\leq i\leq p-1$ is equivalent to $a\mid p-1$. In \cite[Example 3.3]{byottchildselder}, it is proved that the strong condition is implied by $t<\frac{pe}{p-1}-2$, obtaining a slightly weaker result than Theorem \ref{thm:cyclicpfreeness} \ref{thm:cyclicpfreeness3}.
\end{example}

\subsection{Continued fractions and distance to the nearest integer}\label{sec:contfrac}

We will need some notions and results on continued fractions. The reader can consult \cite[Section 2.4]{gildegpdihedral} and the references mentioned therein for further information.

For a non-negative integer $x$ and an odd prime number $p$, we write $$\frac{x}{p}=[a_0;a_1,\dots,a_n]=a_0+\cfrac{1}{a_1+\cfrac{1}{\ddots+\cfrac{1}{a_n}}}$$ for the continued fraction expansion of the rational number $\frac{x}{p}$. By definition, $a_0$ is just its integer part. If $a$ is the remainder of $x$ mod $p$, then $\frac{a}{p}=[0;a_1,\dots,a_n]$. 

For $i\geq0$, the rational number $$\frac{p_i}{q_i}=[a_0;a_1,\dots,a_i],$$ where the fraction is in irreducible form, is commonly known as the $i$-th convergent of $\frac{x}{p}$. For us, the numerators $p_i$ are almost completely irrelevant, while we will need the denominators $q_i$. For this reason, from now on we shall refer to $q_i$ itself as the \textbf{$i$-th convergent}. By definition, $q_0=1$, $q_1=a_1$ and $q_n=p$. Moreover, it is satisfied that $q_{i+2}=a_{i+2}q_{i+1}+q_i$ for every $i\geq2$, and similarly for the numbers $p_i$.

Now, let us denote $$q_{i,r_i}=r_iq_{i+1}+q_i,\quad0\leq r_i\leq a_{i+2},$$ and define integer numbers $p_{i,r_i}$ analogously. The fractions $\frac{p_{i,r_i}}{q_{i,r_i}}$ are usually called semiconvergents or intermediate convergents of $\frac{x}{p}$. Following our convention from the previous paragraph, we will reserve the term \textbf{semiconvergent} for the denominators $q_{i,r_i}$. It is immediate from the definition that $q_{i,0}=q_i$ and $q_{i,a_{i+2}}=q_{i+2}$, so the convergents are semiconvergents. If $i$ is even, we shall say that $q_{i,r_i}$ is an \textbf{even-indexed semiconvergent}.

For a real number $\alpha$, we write $\widehat{\alpha}$ for its fractional part, so that $\alpha=\lfloor\alpha\rfloor+\widehat{\alpha}$. On the other hand, if $2\alpha\notin\mathbb{Z}$, we write $||\alpha||$ for the distance of $\alpha$ to the nearest integer.

\begin{pro}\label{pro:contfrac} Let $\alpha=\frac{a}{p}$ with $0<a<p$ and write $\alpha=[a_0;a_1,\dots,a_n]$ for its continued fraction expansion. Let $\frac{p_i}{q_i}$ be its $i$-th convergent, $0\leq i\leq n$.
\begin{enumerate}[label=\arabic*)]
    \item\label{pro:contfrac1} $||q_{i+1}\frac{a}{p}||<||q_i\frac{a}{p}||$ for every $0\leq i\leq n-1$.
    \item\label{pro:contfrac2} If $q\in\mathbb{Z}$ and $0<q<q_i$, then $||q\frac{a}{p}||\geq||q_{i-1}\frac{a}{p}||$, and in particular, $||q\frac{a}{p}||>||q_i\frac{a}{p}||$.
    \item\label{pro:contfrac3} If $i>0$ is even (resp. odd), then $\widehat{q_i\frac{a}{p}}<\frac{1}{2}$ (resp. $\widehat{q_i\frac{a}{p}}>\frac{1}{2}$).
    \item\label{pro:contfrac4} $||q_{n-1}\frac{a}{p}||=\frac{1}{p}$, so $\widehat{q_{n-1}\frac{a}{p}}=\frac{1}{p}$ if $n$ is odd and $\widehat{q_{n-1}\frac{a}{p}}=\frac{p-1}{p}$ otherwise.
\end{enumerate}
\end{pro}

A proof is given at \cite[Proposition 2.7]{gildegpdihedral}.

\section{Degree $p$ extensions of $p$-adic fields}\label{sec:redtotram}

Let $L/K$ be a degree $p$ extension of $p$-adic fields. Let $\widetilde{L}$ be the normal closure of $L/K$ and let $G=\mathrm{Gal}(\widetilde{L}/K)$. Write $\{G_i\}_{i=-1}^{\infty}$ for the chain of ramification groups of $\widetilde{L}/K$. If $\widetilde{L}/K$ is unramified (in which case necessarily $\widetilde{L}=L$ because a non-Galois extension must be ramified), the ramification jump of $\widetilde{L}/K$ is $t=-1$.

Otherwise, we have by \cite[Chapter IV, Corollary 3]{serre} that $G_1\cong C_p$, the cyclic group with $p$ elements. In that case, the largest ramification jump of $E/K$ is the positive integer $t$ such that $G_t\cong C_p$ and $G_{t+1}$ is trivial. Note that there might be a ramification jump at $0$ depending on whether $\widetilde{L}/L$ is unramified or not, but this is not relevant for our purposes. Thus, we will refer to $t$ as above as the ramification jump of $\widetilde{L}/K$. By \cite[Chapter IV, Proposition 2]{serre}, $t$ is also the ramification jump of $\widetilde{L}/M$. 

Following the definition at Section \ref{sec:prelimpadic}, the Herbrand function of $L/K$ is $\varphi_{L/K}=\varphi_{\widetilde{L}/K}\varphi_{\widetilde{L}/L}$. A straightforward calculation yields $$\ell\coloneqq\varphi_{L/K}(t)=\frac{t}{e(\widetilde{L}:L)},$$ which is the ramification jump of $L/K$ (see \cite[Section 2.2.6]{elder2018}). This invariant is labeled as the phantom ramification number of $L/K$ in \cite[p. 58]{greither1992}. Note that if $\widetilde{L}/K$ is totally ramified, we have that $\ell=\frac{t}{r}$. Let $e$ be the ramification index of $K$ over $\mathbb{Q}_p$. From the proof of \cite[Proposition 2]{berge1978}, it is easy to see that $1\leq t\leq\frac{rpe}{p-1}$. Equivalently, $$0<\ell\leq\frac{pe}{p-1}.$$ In analogy with the Galois case, we use the following terminology:

\begin{defi} Let $L/K$ be a degree $p$ extension of $p$-adic fields with ramification jump $\ell$.
\begin{enumerate}
    \item We say that $L/K$ is \textbf{maximally ramified} if $\ell=\frac{pe}{p-1}$.
    \item We say that $L/K$ is \textbf{almost maximally ramified} if $\ell\geq\frac{pe}{p-1}-1$.
\end{enumerate}
\end{defi}

Under this terminology, the terms for the behavior of $L/K$ described in Theorem \ref{maintheorem} \ref{mainthm3} and \ref{mainthm4} fit with those for the corresponding particular cases when $L/K$ is Galois.

Since the residue field of $K$ is finite, the Galois group $G\coloneqq\mathrm{Gal}(\widetilde{L}/K)$ is solvable \cite[Chapter IV, Section 2, Corollary 5]{serre}. From \cite[(7.5)]{childs}, we obtain that $L/K$ is Hopf-Galois. By \cite[Proposition 4.a)]{childs1989}, $L/K$ admits actually a unique Hopf-Galois structure $H$. In addition, by a theorem of Galois, $G$ is solvable if and only if it is a Frobenius group of degree $p$, that is, $G\cong C_p\rtimes C_r$ with $r\mid p-1$ (see for example \cite[Theorem 1]{ben-shimol}). Fix a presentation \begin{equation}\label{eq:presentG}G=\langle\sigma,\tau\,|\,\sigma^p=\tau^r=1_G,\,\tau\sigma=\sigma^g\tau\rangle,
\end{equation} where $g\in\mathbb{Z}$ reduces mod $p$ to an order $r$ element of $(\mathbb{Z}/p\mathbb{Z})^{\times}$. Calling $J=\langle\sigma\rangle$ and $G'=\langle\tau\rangle$, we have that $G=J\rtimes G'$ with $J\cong C_p$ and $G'\cong C_r$. Since $J$ is a normal subgroup of $G$, it is the only order $p$ subgroup of $G$. Hence $M\coloneqq\widetilde{L}^J$ is the only intermediate field of $\widetilde{L}/K$ with degree $r$ over $K$, and $J=\mathrm{Gal}(\widetilde{L}/M)$. Let us choose the element $\tau$ at \eqref{eq:presentG} such that $L=\widetilde{L}^{G'}$. Then, $L/K$ is almost classically Galois with complement $M$.

\subsection{Reduction to the totally ramified case}

We shall show that in order to prove Theorem \ref{maintheorem} we can assume without loss of generality that the normal closure $\widetilde{L}$ of a degree $p$ extension of $p$-adic fields $L/K$ is totally ramified. In order to do so, we will replace $L/K$ by another degree $p$ extension of $p$-adic fields whose ring of integers follows the same criteria for freeness over the associated order and whose normal closure is totally ramified. Namely, such an extension is $L'/K'$, where $L'$ (resp. $K'$) is the inertia field of $\widetilde{L}/L$ (resp. $\widetilde{L}/K$). 


Note that by definition of inertia field, $L'/L$ (resp. $K'/K$) is the largest unramified subextension of $\widetilde{L}/L$ (resp. $\widetilde{L}/K$), and therefore $\widetilde{L}/L'$ (resp. $\widetilde{L}/K'$) is totally ramified. The idea behind the fact that the criteria for the freeness does not vary from $L/K$ to $L'/K'$ is that we can \textit{ignore} the unramified subextensions. This is the content of the following technical lemma.

\begin{lema}\label{lem:tensor} Let $L/K$ be an $H$-Galois extension of $p$-adic fields. Let $K'/K$ be another extension which is arithmetically disjoint with $L/K$ (i.e, the discriminants of $L/K$ and $K'/K$ are coprime) and let $L'=LK'$. Call $H'=K'\otimes_KH$, which is a Hopf-Galois structure on $L'/K'$. Then $\mathfrak{A}_{H'}=\mathfrak{A}_H\otimes_{\mathcal{O}_K}\mathcal{O}_{K'}$ and $\mathcal{O}_{L'}$ is $\mathfrak{A}_{H'}$-free if and only if $\mathcal{O}_L$ is $\mathfrak{A}_H$-free.
\end{lema}
\begin{proof}
Since $L/K$ and $K'/K$ are arithmetically disjoint, $\mathcal{O}_{L'}=\mathcal{O}_L\otimes_{\mathcal{O}_K}\mathcal{O}_{K'}$ (see \cite[Chapter III, (2.13)]{frohlichtaylor}).

It is clear that $\mathfrak{A}_H\otimes_{\mathcal{O}_K}\mathcal{O}_{K'}\subseteq H\otimes_KK'=H'$. Since $\mathfrak{A}_H\otimes_{\mathcal{O}_K}\mathcal{O}_{K'}$ acts $\mathcal{O}_K$-linearly on $\mathcal{O}_{L'}=\mathcal{O}_L\otimes_{\mathcal{O}_K}\mathcal{O}_{K'}$ componentwise, $\mathfrak{A}_H\otimes_{\mathcal{O}_K}\mathcal{O}_{K'}\subseteq\mathfrak{A}_{H'}$. Conversely, let $h\in\mathfrak{A}_{H'}$, so that $h\in H'=H\otimes_KK'$. Let $\{z_j\}_{j=1}^s$ be an $\mathcal{O}_K$-basis of $\mathcal{O}_{K'}$. Then, there are unique elements $h_1,\dots,h_s\in H$ such that $$h=\sum_{j=1}^sh_jz_j.$$ It is enough to prove that $h_j\in\mathfrak{A}_H$ for all $1\leq j\leq s$. Indeed, given $\gamma\in\mathcal{O}_L\subseteq\mathcal{O}_{L'}$, we have that $h\cdot\gamma\in\mathcal{O}_L'$. Now, $$h\cdot\gamma=\sum_{j=1}^s(h_j\cdot\gamma)z_j.$$ Since $\mathcal{O}_L$ is $\mathcal{O}_K$-flat, $\{z_j\}_{j=1}^s$ is an $\mathcal{O}_L$-basis of $\mathcal{O}_{L'}$. Therefore, it follows that $h_j\cdot\gamma\in\mathcal{O}_L$ for all $1\leq j\leq s$, proving that $h_j\in\mathfrak{A}_H$ as desired.

As for the freeness, first assume that $L/K$ is unramified. Then $\mathcal{O}_L$ is $\mathfrak{A}_H$-free, and using that $\mathcal{O}_{L'}=\mathcal{O}_L\otimes_{\mathcal{O}_K}\mathcal{O}_{K'}$ we obtain that $\mathcal{O}_{L'}$ is $\mathfrak{A}_{H'}$-free. Otherwise, since $L/K$ and $K'/K$ are arithmetically disjoint, necessarily $K'/K$ is unramified. Hence $\pi_K\mathcal{O}_{K'}=\pi_{K'}\mathcal{O}_{K'}\neq\mathcal{O}_{K'}$, and then $\mathcal{O}_{K'}$ is $\mathcal{O}_K$-faithfully flat. Therefore, $\mathcal{O}_L$ is $\mathfrak{A}_H$-free if and only if $\mathcal{O}_L\otimes_{\mathcal{O}_K}\mathcal{O}_{K'}$ is $\mathfrak{A}_H\otimes_{\mathcal{O}_K}\mathcal{O}_{K'}$-free, that is, $\mathcal{O}_{L'}$ is $\mathfrak{A}_{H'}$-free.
\end{proof}

\begin{rmk}\normalfont Lemma \ref{lem:tensor} is a direct generalization of \cite[Proposition 5.18 and Corollary 5.19]{gilrioinduced}, that are written in the context that the compositum $L'/K$ is a Galois extension with Galois group isomorphic to a semidirect product (so that it admits induced Hopf-Galois structures). This is completely irrelevant for the description of the associated order in $H'$, and hence the proof is the same. As for the case of the freeness, our proof using the faithful flatness of $\mathcal{O}_{K'}$ is also valid in that case.
\end{rmk}

We are ready to validate the claim of this section.

\begin{pro}\label{pro:redtotram} Let $L/K$ be a totally ramified degree $p$ extension of $p$-adic fields. Let $\widetilde{L}$ be the normal closure of $L/K$ and let $M$ be the Galois complement of $L/K$. Let $L'$ (resp. $K'$) be the inertia field of $\widetilde{L}/L$ (resp. $\widetilde{L}/K$). Then: 
\begin{enumerate}[label=\arabic*)]
    \item\label{pro:redtotram1} $L'/K'$ is a degree $p$ extension of $p$-adic fields with normal closure $\widetilde{L}$ (so it is Hopf-Galois).
    \item\label{pro:redtotram2} $\widetilde{L}/K'$ is totally ramified.
    \item\label{pro:redtotram3} $L'/K'$ is almost classically Galois with complement $M$.
    \item\label{pro:redtotram4} $H'\coloneqq H\otimes_KK'$ is the only Hopf-Galois structure on $L'/K'$.
    \item\label{pro:redtotram5} $\mathfrak{A}_{H'}=\mathcal{O}_{K'}\otimes_{\mathcal{O}_K}\mathfrak{A}_H$ and $\mathcal{O}_{L'}$ is $\mathfrak{A}_{H'}$-free if and only if $\mathcal{O}_L$ is $\mathfrak{A}_H$-free.
\end{enumerate}
\end{pro}
\begin{proof}
Let us call $G=\mathrm{Gal}(\widetilde{L}/K)$, $G'=\mathrm{Gal}(\widetilde{L}/L)$ and $J=\mathrm{Gal}(\widetilde{L}/M)$, so that $G=J\rtimes G'$.

Since $L\cap M=K$ and $K'\subseteq M$, we have $L\cap K'=K$. Moreover, since $K'/K$ is unramified, it is Galois. Therefore we can apply \cite[Theorem 5.5]{cohn1991}, obtaining that $L$ and $K'$ are $K$-linearly disjoint.

\begin{enumerate}
    \item It is trivial that the normal closure of $L'/K'$ is $\widetilde{L}$. Let us prove that $[L':K']=p$. We write $G_0$ (resp. $G_0'$) for the inertia group of $G$ (resp. $G'$), so that $L'=\widetilde{L}^{G_0'}$ and $K'=\widetilde{L}^{G_0}$. Since $G_0'\subseteq G_0$, we have $K'\subseteq L'$. Applying \cite[Chapter IV, Proposition 2]{serre}, we obtain that $G_0'=G_0\cap G'$. Then, $$LK'=\widetilde{L}^{G'}\widetilde{L}^{G_0}=\widetilde{L}^{G'\cap G_0}=\widetilde{L}^{G_0'}=L'.$$ Since $L/K$ and $K'/K$ are linearly disjoint, $[L':K']=[L:K]=p$, as we wanted.

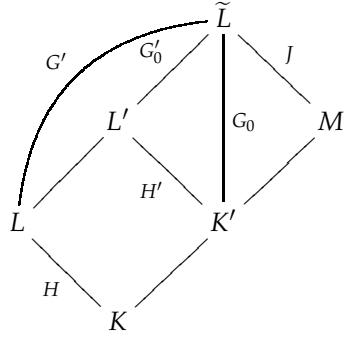
\[
\xymatrix{
&& \ar@{-}@/_2pc/[ddll]_{G'} \ar@{-}[dl]_{G_0'} \widetilde{L} \ar@{-}[dd]^{G_0} \ar@{-}[dr]^J & \\
& \ar@{-}[dl] L' \ar@{-}[dr]_{H'} & & \ar@{-}[dl] M \\
L \ar@{-}[dr]_H & & \ar@{-}[dl] K' & \\
& K & &
}
\]

    \item Since the extension $L/K$ is totally ramified and linearly disjoint with $K'/K$, $L'/K'$ is totally ramified. On the other hand, by definition of inertia field, $\widetilde{L}/L'$ is totally ramified. Therefore, $\widetilde{L}/K'$ is totally ramified.
    
    \item Arguing as in \ref{pro:redtotram2}, since $L/K$ is linearly disjoint with $M/K$, $\widetilde{L}/M$ is totally ramified. It follows that $K'\subseteq M$. Moreover, since $M/K$ is Galois, $M/K'$ also is. Then the fundamental theorem of Galois theory applied to the extension $\widetilde{L}/K$ gives that $J$ is a normal subgroup of $G_0$. Since $J\cap G'=\{1_G\}$ and $G_0'\subseteq G'$, we have $J\cap G_0'=\{1_G\}$. Hence, the product of $J$ and $G_0'$ within $G_0$ is semidirect. Now, since $L/K$ is totally ramified, $$|J\rtimes G_0'|=|J||G_0'|=[L:K]e(\widetilde{L}/L)=e(\widetilde{L}/K)=|G_0|.$$ It follows that $G_0=J\rtimes G_0'$. Hence the extension $L'/K'$ is almost classically Galois with complement $M$.

    \item Since $L'/K'$ has degree $p$, it posseses a unique Hopf-Galois structure $H'$. On the other hand, since $H$ is a Hopf-Galois structure on $L/K$, $H\otimes_KK'$ is a Hopf-Galois structure on $L'/K'$. Necessarily, $H'=H\otimes_KK'$.

    \item We already know that $L$ and $K'$ are $K$-linearly disjoint and $LK'=L'$. Since $L/K$ is totally ramified and $K'/K$ is unramified, they are $K$-arithmetically disjoint. Then, the result follows from applying Lemma \ref{lem:tensor}.
\end{enumerate}

\end{proof}

\begin{coro} Let $L/K$ be a degree $p$ extension of $p$-adic fields, let $H$ be its only Hopf-Galois structure and let $\widetilde{L}$ be its normal closure.
\begin{enumerate}[label=\arabic*)]
    \item If $\widetilde{L}/L$ is unramified, then the $\mathfrak{A}_{L/K}$-freeness of $\mathcal{O}_L$ follows the same criteria as in the Galois case \cite{bertrandiasbertrandiasferton,bertrandiasferton}.
    \item If $\widetilde{L}/L$ has ramification index equal to $2$, then the $\mathfrak{A}_{L/K}$-freeness of $\mathcal{O}_L$ follows the same criteria as in the dihedral case \cite{gildegpdihedral}.
\end{enumerate}
\end{coro}

\section{The maximally ramified case}\label{sec:maxram}

From now on, we assume that $\widetilde{L}/K$ is totally ramified, as by Proposition \ref{pro:redtotram}, there is no loss of generality in doing so. Write $t$ for the ramification jump of $\widetilde{L}/K$. Then, the ramification jump of $L/K$ is $\ell=\frac{t}{r}$. Let $a$ be the residue class of $\ell$ as a $p$-adic integer.

In this section, we will prove Theorem \ref{maintheorem} \ref{mainthm1}. Let us assume that that $a=0$. Since $\ell=\frac{t}{r}$, the remainder of $t$ mod $p$ also vanishes, so by \cite[Chapter III, \textsection2, Proposition 2.3]{fesenkovoskotov}, we have $t=\frac{rpe}{p-1}$. Then $\ell=\frac{pe}{p-1}$ and $L/K$ is maximally ramified.

The author already proved that, in this situation, $\mathcal{O}_L$ is $\mathfrak{A}_{L/K}$-free \cite[Proposition 7.10]{gilkummer}. Actually, by \cite[Corollary 2.5.6]{trumanthesis}, this is implied by $\mathfrak{A}_{L/K}$ being the maximal $\mathcal{O}_K$-order in $H$. Thus, for our purposes it is enough to establish the maximality of $\mathfrak{A}_{L/K}$.

Let $M$ be as in Section \ref{sec:redtotram}, so that $\widetilde{L}/M$ is a cyclic degree $p$ extension. Recall that the ramification jump $t$ of $\widetilde{L}/K$ is also the ramification jump of $\widetilde{L}/M$. Hence, the extension $\widetilde{L}/M$ satisfies the conditions of Theorem \ref{thm:cyclicpfreeness} \ref{thm:cyclicpfreeness1}, whose proof relies on some nice properties satisfied by the extension. We shall adapt these properties to the extension $L/K$ in order to validate Theorem \ref{maintheorem} \ref{mainthm1}. In the end, the proof will be analogous to the one for the dihedral case \cite[Section 5]{gildegpdihedral}.

By \cite[Chapter III, Section 2, Proposition 2.3]{fesenkovoskotov}, we have that $\zeta_p\in M$ and $\widetilde{L}=M(\gamma)$ for some $\gamma\in\mathcal{O}_{\widetilde{L}}$ with $v_{\widetilde{L}}(\gamma)=1$ and $\gamma^p\in\mathcal{O}_M$. In particular, $\widetilde{L}/M$ is a Kummer extension with Kummer generator $\gamma$. Then, $\sigma(\gamma)=\zeta_p\gamma$ for some choice of primitive $p$-th root of unity $\zeta_p$, whence \begin{equation}\label{eq:galoiseig}
    \sigma^i(\gamma^j)=\zeta_p^{ij}\gamma^j
\end{equation} for every $0\leq i\leq p-1$. We see that the elements $\gamma^j$ are eigenvectors for the Galois action of $\widetilde{L}/M$. Moreover, they form a generating system for $\widetilde{L}/M$. In \cite[Proposition 4.5]{gilkummer}, it is proved that this is actually a characterization for the notion of Kummer extension.


Let $\alpha=N_{\widetilde{L}/L}(\gamma)\in\mathcal{O}_L$. We shall see that this element is an analogue of $\gamma$ for the extension $L/K$. In \cite[Lemma 7.12]{gilkummer} it is proved that $v_L(\alpha)=1$ (in particular, $L=K(\alpha)$), $\alpha^p\in\mathcal{O}_K$ and $M=K(\zeta_p)$. Since $L\cap M=K$ and $LM=\widetilde{L}$, $L/K$ is an almost classically Galois extension with complement $M$, whence $H=M[J]^{G'}$. In particular, the elements of $H$ are $M$-linear combinations of elements of $J$. Fix a $K$-basis $\{w_i\}_{i=0}^{p-1}$ of $H$. From \eqref{eq:galoiseig} we deduce that for each $0\leq i,j\leq p-1$ there is $\lambda_{ij}\in K$ such that $w_i\cdot\alpha^j=\lambda_{ij}\alpha^j$, that is, the elements $\alpha^j$ are eigenvectors for the action of $H$ on $L$. In the language of \cite[Definition 5.1]{gilkummer}, $\alpha^j$ is an $H$-eigenvector for all $0\leq j\leq p-1$. What is more, these elements form a generating system of $H$-eigenvectors for $L/K$. Following the terminology at \cite[Definition 5.2]{gilkummer}, $L/K$ is $H$-Kummer. Actually, what we have obtained is a particular case of \cite[Proposition 6.3]{gilkummer}.

Now, the condition that $v_L(\alpha)=1$ means that $\{\alpha^j\}_{j=0}^{p-1}$ is an integral basis for $L/K$. Hence, we can apply \cite[Theorem 1.2 1]{gilkummer} to this case, giving rise to the following result:

\begin{pro} Let $\Lambda=(\lambda_{ij})_{i,j=0}^{p-1}$ and let $\Omega=(\omega_{li})_{l,i=0}^{p-1}$ be the inverse of $\Lambda^t$. Then the associated order $\mathfrak{A}_{L/K}$ of $\mathcal{O}_L$ in $H$ admits an $\mathcal{O}_K$-basis formed by $$v_i=\sum_{l=0}^{p-1}\omega_{li}w_l,\quad 0\leq i\leq p-1.$$ Moreover, this basis is a system of primitive pairwise orthogonal idempotents for $H$.
\end{pro}

With this, we can prove our claim.

\begin{coro} $\mathfrak{A}_{L/K}$ is the maximal $\mathcal{O}_K$-order in $H$.
\end{coro}
\begin{proof}
Since $\{v_i\}_{i=0}^{p-1}$ forms a primitive system of pairwise orthogonal idempotents in $H$, the product of elements of $H$ written with respect to this system is componentwise. Hence, the map $$\begin{array}{rccl}
    \varphi\colon & H & \longrightarrow & K^p, \\
     & v_i & \longmapsto & e_i\coloneqq(\delta_{ij})_{j=0}^{p-1}.
\end{array}$$ is an isomorphism of $K$-algebras. Now, the maximal $\mathcal{O}_K$-order in $K^p$ is clearly $\mathcal{O}_K^p$, and its inverse image by $\varphi$ is the maximal $\mathcal{O}_K$-order in $H$. But this inverse image is just the set of all $\mathcal{O}_K$-linear combinations of the $v_i$, that is, $\mathfrak{A}_{L/K}$. Then $\mathfrak{A}_{L/K}$ is the maximal $\mathcal{O}_K$-order in $H$.
\end{proof}

\begin{rmk}\normalfont The proof is essentially the same as the one at \cite[Section 5]{gildegpdihedral} for the case in which $\widetilde{L}/K$ is dihedral of degree $2p$. In that situation, we found explicitly the values of the $H$-eigenvalues $\lambda_{ij}$ because we were aware of the explicit form of a $K$-basis of $H$. In our case, determining the eigenvalues could be possible as well by making explicit the $K$-basis $\{w_i\}_{i=0}^{p-1}$, but it is useless for our purposes.
\end{rmk}

\section{Typical degree $p$ extensions}\label{sec:typicalp}

Following Elder \cite{elder2018}, we will say that a degree $p$ extension of $p$-adic fields $L/K$ is typical if it is not generated by a $p$-th root of a uniformizer $\pi_K$. In this case, the ramification jump $\ell$ of $L/K$ satisfies $0<\ell<\frac{pe}{p-1}$. Equivalently, $a\neq0$, where $a$ is the residue class of $\ell$ mod $p$. For this class of extensions, Elder provided a description which we recall (and translate to our notation) for later use.

\begin{thm}[Elder]\label{thm:eldertypicalp} Let $p$ be an odd prime. \begin{enumerate}[label=\arabic*)]
    \item Let $L/K$ be a typical degree $p$ extension of $p$-adic fields with normal closure $\widetilde{L}$ and assume that $\widetilde{L}/K$ is totally ramified. Call $r=[\widetilde{L}:L]$. Then there are $x\in L$ with $L=K(x)$, $\alpha,\beta\in K$ and integers $b,c\in\mathbb{Z}$ with $0\leq c<r$ and $\gcd(c,r)=1$ satisfying \begin{equation}\label{eq:typicalx}
        x^p-\alpha^{\frac{p-1}{r}}x-\beta=0,
    \end{equation} with $v_K(\alpha)=c$ and $v_K(\beta)=-b$. Moreover, the ramification jump of $L/K$ is $\ell=b+\frac{pc}{r}$.
    \item\label{thm:eldertypicalp2} Let $K$ be a $p$-adic field with absolute ramification index $e$ and set positive integers $r,t$ with $p\nmid t$, $r\mid p-1$ and $0<\frac{t}{r}<\frac{pe}{p-1}$. Let $c$ be the remainder of $t$ mod $r$ and let $b\coloneqq\frac{t-pc}{r}$ (note that $b$ is integer because $p\equiv1\,(\mathrm{mod}\,r)$ and $t\equiv c\,(\mathrm{mod}\,r)$). Let $\alpha,\beta\in K$ such that $v_K(\alpha)=c$, $v_K(\beta)=-b$ and let $x$ be an element in an algebraic closure of $K$ satisfying an equation of the form \eqref{eq:typicalx}. Then $L=K(x)$ is a typical degree $p$ extension of $K$ with totally ramified normal closure $\widetilde{L}$ such that $[\widetilde{L}:L]=r$ and the ramification jump of $L/K$ is $\ell\coloneqq\frac{t}{r}$.
\end{enumerate}

In this situation, $\widetilde{L}=L(y)$ with $y^r=\alpha$, and we can choose the morphisms $\sigma$ and $\tau$ from \eqref{eq:presentG} to be defined by $$\begin{array}{rcclcrccl}
    \sigma\colon & x & \mapsto & x+y(1+\eta), & & \tau\colon & x & \mapsto & x,\\
     & y & \mapsto & y, & & & y & \mapsto & \zeta_ry,
    \end{array}$$ where $\eta\in\widetilde{L}$ is an element with valuation $v_{\widetilde{L}}(\eta)=rpe-(p-1)t$ and $\zeta_r=\chi(g)$, where $\chi\colon\mathbb{Z}/p\mathbb{Z}\longrightarrow\mathbb{Z}_p$ is the $p$-adic Teichmuller character.
\end{thm}

This result is \cite[Theorem 2.2]{elder2018}. Actually, the statement therein is more general than Theorem \ref{thm:eldertypicalp} because it includes the case $\mathrm{char}(K)=p$ and it is not assumed that $\widetilde{L}/K$ is totally ramified (under Elder's notation, our restrictions correspond to $e=d$ and $f=1$). For the reader's guidance, we include a dictionary of notation with respect to that work.

\begin{center}
\begin{tabular}{|c|c|} \hline
    Notation in Theorem \ref{thm:eldertypicalp} & Notation in \cite[Theorem 2.2]{elder2018} \\ \hline
    $r$ & $d$ \\ \hline
    $c$ & $t$ \\ \hline
    $t$ & $e\ell$ \\ \hline
    $\zeta_r$ & $\rho$ \\ \hline
    $g$ & $r$ \\ \hline
\end{tabular}
\end{center}

These refer to the cases where a different label has been taken. On the contrary, the labels $b$ and $\ell$ have the same meaning in Theorem \ref{thm:eldertypicalp} and \cite[Theorem 2.2]{elder2018}.

\begin{rmk}\normalfont\label{rmk:valx} Let $L/K$ be a typical degree $p$ extension of $p$-adic fields and let $x\in L$ be an element satisfying \eqref{eq:typicalx}. It can be seen from the proof of Theorem \ref{thm:eldertypicalp} at \cite[Section 2.2.4]{elder2018} that $v_L(x)=-b$. The idea is that, calling $X$ the Artin-Schreier generator of the Galois degree $p$ extension $\widetilde{L}/M$, the element $x_1=\frac{1}{r}Tr_{\widetilde{L}/L}(yX)$ is the first term of a Cauchy sequence $\{x_n\}$ (with respect to the extension of $p$-adic absolute value to $L$) with limit $x$. It is proved that all these terms have $L$-valuation $-b$, from which it follows that so does $x$.
\end{rmk}

\begin{rmk}\normalfont\label{rmk:typicalgalois} Let $L/K$ be a typical degree $p$ extension with ramification jump $\ell=b+\frac{pc}{r}$. Since $t\equiv c\,(\mathrm{mod}\,r)$ and $0\leq c<r$, we have that $c$ is the remainder of $t$ mod $r$. The property that $\gcd(c,r)=1$ ensures that $c=0$ if and only if $r=1$, that is, $L/K$ is Galois.
\end{rmk}

\subsection{The only Hopf-Galois structure}

From now on, $L/K$ will be a typical degree $p$ extension of $p$-adic fields. In this section, we describe the only Hopf-Galois structure $H$ on $L/K$ (both the underlying Hopf algebra and its action on the top field).

\subsubsection{The underlying Hopf algebra}

The description of the underlying Hopf algebra of the unique Hopf-Galois structure on a separable degree $p$ extension was carried out by Childs \cite[Section 2]{childs1989} under the assumption that the ground field has zero characteristic. Elder showed that the same proof by Childs is valid for typical extensions without any restriction on the characteristic \cite[Theorem 3.1]{elder2018}. The result essentially states that one can find explicitly $\Psi\in H$ such that $H=K[\Psi]$ and $\Psi^p=\lambda\Psi$ for some $\lambda\in\mathcal{O}_K$, and relies in a classification for rank $p$ Hopf algebras that arises from specializing the classification for group schemes of prime order by Tate and Oort \cite{tateoort} (see \cite[Chapter 4]{childs}).

In this part, we shall reproduce this explicit description of $H$. Even though the contents are presented in a slightly different way, our proofs are essentially equivalent to the ones at the references above. In addition, we will go a step further by finding the explicit form of the degree $p$ equation satisfied by the generator $\Psi$, which will be extensively used in Sections \ref{sec:freescaffold} and \ref{sec:freecontfrac}.

We start reviewing the classification of rank $p$ Hopf algebras. Let $\mathfrak{H}$ be an $\mathcal{O}_K$-Hopf algebra with rank $p$. Let us write $m_{\mathfrak{H}}$ for the multiplication of $\mathfrak{H}$, $\Delta_{\mathfrak{H}}$ for its comultiplication, $\epsilon_{\mathfrak{H}}$ for its counity and $S_{\mathfrak{H}}$ for its coinverse. Given $n\in\mathbb{Z}$, let us define a map $[n]\colon\mathfrak{H}\longrightarrow\mathfrak{H}$ as follows:

\begin{itemize}
    \item $[0]=\epsilon_{\mathfrak{H}}$.
    \item $[1]=\mathrm{Id}_{\mathfrak{H}}$ and $[n+1]=m_{\mathfrak{H}}([n]\otimes_K\mathrm{Id}_{\mathfrak{H}})\Delta_{\mathfrak{H}}$ for $n\in\mathbb{Z}_{>0}$.
    \item $[-1]=S_{\mathfrak{H}}$ and $[-n-1]=m_{\mathfrak{H}}([-n]\otimes_K S_{\mathfrak{H}})$ for $n\in\mathbb{Z}_{>0}$.
\end{itemize}

It turns out that $[n]$ is an endomorphism of $\mathcal{O}_K$-Hopf algebras for every $n\in\mathbb{Z}$. Hence, this assignment defines a map $[\cdot]\colon\mathbb{Z}\longrightarrow\mathrm{End}_{\mathcal{O}_K\hbox{-Hopf}}({\mathfrak{H}})$, which is a ring homomorphism, with respect to the underlying ring structure on $\mathfrak{H}$ and the ring structure with convolution and composition on $\mathrm{End}_{\mathcal{O}_K\hbox{-Hopf}}({\mathfrak{H}})$ \cite[(16.1), (1)]{childs}. Since $[p]=[0]$ \cite[(16.1), (3)]{childs}, passage to the quotient gives a homomorphism $\mathbb{Z}/p\mathbb{Z}\longrightarrow\mathrm{End}_{\mathcal{O}_K\hbox{-Hopf}}({\mathfrak{H}})$, which by abuse of notation we also denote by $[\cdot]$. 

Let $\chi\colon\mathbb{Z}/p\mathbb{Z}\longrightarrow\mathbb{Z}_p$ be the $p$-adic Teichmuller character (that is, the unique multiplicative section of the projection $\pi\colon\mathbb{Z}_p\longrightarrow\mathbb{Z}/p\mathbb{Z}$). Let $\mathfrak{H}^+=\mathrm{Ker}(\epsilon_{\mathfrak{H}})$ and for $1\leq i\leq p-1$ call $$\mathfrak{H}_i=\{h\in \mathfrak{H}^+\,\mid\,[m](h)=\chi^i(h)m\hbox{ for all }m\in\mathbb{Z}/p\mathbb{Z}\}.$$ Then $\mathfrak{H}^+=\bigoplus_{i=1}^{p-1}\mathfrak{H}_i$ and it has the structure of a graded ring, that is, $\mathfrak{H}_i\mathfrak{H}_j\subseteq\mathfrak{H}_{i+j}$ for all $i,j\in\mathbb{Z}/p\mathbb{Z}$. This expression is commonly known as an eigenspace decomposition for $\mathfrak{H}$. Since $\mathfrak{H}$ has rank $p$, it is immediate that $\mathfrak{H}_i$ has rank one for every $i$. The following result reduces the problem of describing $\mathfrak{H}$ to finding a generator of $\mathfrak{H}_1$.

\begin{pro}[\cite{childs}, (16.13)]\label{pro:strgroupalg} Let $\mathfrak{H}$ be a rank $p$ $\mathcal{O}_K$-Hopf algebra. Suppose that $\psi\in\mathfrak{H}^+$ is such that $\mathfrak{H}_1=\mathcal{O}_K\psi$. Then $\mathfrak{H}=\mathcal{O}_K[\psi]$ with $\psi^p=s\psi$ for some $s\in\mathcal{O}_K$.
\end{pro}

Let us consider the case in which $\mathfrak{H}=\mathcal{O}_K[\Gamma]$ for some order $p$ group $\Gamma$. Let $\sigma$ be a generator of $\Gamma$ as a cyclic group. Then $[n](\sigma)=\sigma^n$ for all $n\in\mathbb{Z}/p\mathbb{Z}$ \cite[(16.1), (2)]{childs}. There is a complete description of the spaces $\mathfrak{H}_i$ available for this case:

\begin{pro} Let $\mathfrak{H}=\mathcal{O}_K[\Gamma]$, where $K$ is a $p$-adic field and $\Gamma$ is an order $p$ group. Given $1\leq i\leq p-1$, we have $\mathfrak{H}_i=\mathcal{O}_K\psi_i$, where $$\psi_i=-\sum_{m=1}^{p-1}\chi^{-i}(m)\sigma^m,\quad1\leq i\leq p-2,$$ $$\psi_{p-1}=(p-1)\mathrm{Id}-\sum_{m=1}^{p-1}\sigma^m,\quad i=p-1.$$ Moreover, $\mathfrak{H}=\mathcal{O}_K[\psi_1]$.
\end{pro}
\begin{proof}
A complete proof is given at \cite[(16.5),(16.6)]{childs} for $\mathfrak{H}=\mathbb{Z}_p[\Gamma]$. For the general case, we note that $\mathcal{O}_K[\Gamma]=\mathcal{O}_K\otimes_{\mathbb{Z}_p}\mathbb{Z}_p[\Gamma]$ and $\mathbb{Z}_p[\Gamma]_i\subseteq\mathcal{O}_K[\Gamma]_i$ for each $1\leq i\leq p-1$, so we obtain $\mathcal{O}_K[\Gamma]_i=\mathcal{O}_K\otimes_{\mathbb{Z}_p}\mathbb{Z}_p[\Gamma]_i$. Since $\mathcal{O}_K$ is flat as a $\mathbb{Z}_p$-module, the statement follows.
\end{proof}

We can easily give a proof of Proposition \ref{pro:strgroupalg} when $K=\mathbb{Q}_p$. Since $\mathfrak{H}^+$ is a graded ring as a direct sum of $\mathfrak{H}_1,\dots,\mathfrak{H}_{p-1}$, calling $\psi=\psi_1$, we have that for each $1\leq i\leq p$ there is an element $s_i\in\mathbb{Z}_p$ such that $\psi^i=s_i\psi_i$ if $i<p$ and $\psi^p=s_p\psi$. What is more, in this case we are able to determine the numbers $s_i$ mod $p$.

\begin{pro}[\cite{childs}, (16.8)]\label{pro:bimodp} Under the previous notation, we have:
\begin{itemize}
    \item $s_i\equiv i!\,(\mathrm{mod}\,p)$ for all $1\leq i\leq p-1$.
    \item $s_p=ps_{p-1}$.
\end{itemize} In particular $v_p(s_i)=0$ for all $1\leq i\leq p-1$ and $v_p(s_p)=1$.
\end{pro}

In fact, Proposition \ref{pro:bimodp} remains valid after dropping the restriction that $K=\mathbb{Q}_p$. Concretely, if $\mathfrak{H}$ is any $\mathcal{O}_K$-Hopf algebra of rank $p$, since $\mathcal{O}_K[\Gamma]_i=\mathcal{O}_K\otimes_{\mathbb{Z}_p}\mathbb{Z}_p[\Gamma]_i$ for each $0\leq i\leq p-1$, we have that $\mathfrak{H}_i=s_i\mathcal{O}_K$, where the $s_i$ are as in Proposition \ref{pro:bimodp}.

We summarize the information above in the following:

\begin{pro}\label{pro:rankpsummary} Let $\mathfrak{H}$ be a rank $p$ $\mathcal{O}_K$-Hopf algebra. Suppose that $\psi\in\mathfrak{H}^+$ is such that $\mathfrak{H}_1=\mathcal{O}_K\psi$. Then $\mathfrak{H}=\mathcal{O}_K[\psi]$ with $\psi^p=s\psi$ for some $s\in\mathcal{O}_K$. Moreover, if $\mathfrak{H}=\mathcal{O}_K[\Gamma]$ for an order $p$ group $\Gamma=\langle\sigma\rangle$, then we can take $$\psi=-\sum_{m=1}^{p-1}\chi^{-1}(m)\sigma^m,$$ and in that case $s\in\mathbb{Z}_p$ with $v_p(s)=1$.
\end{pro}

We apply these techniques to describe the underlying Hopf algebra at the only Hopf-Galois structure on $L/K$ as follows.

\begin{thm}\label{thm:hopfalgdimp} Let $L/K$ be a typical degree $p$ Hopf-Galois extension of $p$-adic fields and take the notations from Theorem \ref{thm:eldertypicalp}. Let $H$ be the only Hopf-Galois structure on $L/K$. Then $H=K[\Psi]$, where $$\Psi=-\frac{1}{y}\Big(\sum_{m=1}^{p-1}\chi(m)^{-1}\sigma^m\Big).$$
\end{thm}
\begin{proof}
Let $J$ be defined as in the paragraph after \eqref{eq:presentG}. Since $J$ is a group of order $p$, we can apply Proposition \ref{pro:rankpsummary} with $\mathfrak{H}=\mathcal{O}_M[J]$. Thus, for $\psi=-\sum_{m=1}^{p-1}\chi(m)^{-1}\sigma^m$ we obtain $\mathfrak{H}=\mathcal{O}_M[\psi]$. It is then immediate that $M[J]=M[\psi]$. Now, note that $\Psi=\frac{1}{y}\psi$, and since $y\in M^{\times}$, we have $M[J]=M[\Psi]$ as well. Let us prove that $H=K[\Psi]$. First, we prove that $\Psi\in H=M[J]^{G'}$. Indeed:
\begin{equation*}
    \begin{split}
        \tau\cdot\Psi&=-\frac{1}{\zeta_ry}\Big(\sum_{m=1}^{p-1}\chi(m)^{-1}\sigma^{gm}\Big)\\&=-\frac{1}{\zeta_ry}\Big(\sum_{m=1}^{p-1}\chi(g^{-1}m)^{-1}\sigma^m\Big)\\&=-\frac{1}{\zeta_ry}\Big(\sum_{m=1}^{p-1}\chi(g)\chi(m)^{-1}\sigma^m\Big)\\&=-\frac{1}{\zeta_ry}\zeta_r\Big(\sum_{m=1}^{p-1}\chi(m)^{-1}\sigma^m\Big)=\Psi.
    \end{split}
\end{equation*} Thus, $K[\Psi]\subseteq H$. Now, \begin{equation*}
    \begin{split}
        \mathrm{dim}_K(M[\Psi])=r\mathrm{dim}_M(M[\Psi])=r\mathrm{dim}_M(M[J])=rp.
    \end{split}
\end{equation*} But also $$\mathrm{dim}_K(M[\Psi])=\mathrm{dim}_K(M\otimes_KK[\Psi])=r\mathrm{dim}_K(K[\Psi]).$$ Hence $K[\Psi]$ has $K$-dimension $p$, whence $H=K[\Psi]$ follows.
\end{proof}

From now on, we reserve the label $\Psi$ for the element at Theorem \ref{thm:hopfalgdimp}. We can now derive the explicit form of the degree $p$ polynomial satisfied by $\Psi$.

\begin{coro}\label{coro:eqdegp} Let $L/K$ be a typical degree $p$ Hopf-Galois extension of $p$-adic fields with Hopf-Galois structure $H=K[\Psi]$. Then there is a unit $\varepsilon\in\mathbb{Z}_p^{\times}$ such that the element $$\lambda\coloneqq\frac{\varepsilon p}{y^{p-1}}\in\mathcal{O}_K$$ satisfies $\Psi^p=\lambda\Psi$.
\end{coro}
\begin{proof}
Let $\psi=-\sum_{m=1}^{p-1}\chi^{-1}(m)\sigma^m$. By Proposition \ref{pro:rankpsummary}, there is $s\in\mathbb{Z}_p$ with $v_p(s)=1$ such that $\psi^p=s\psi$. Then, $$\Psi^p=\frac{1}{y^p}\psi^p=\frac{s}{y^p}\psi=\frac{s}{y^{p-1}}\Psi,$$ and since $v_p(s)=1$, there is $\varepsilon\in\mathbb{Z}_p^{\times}$ such that $s=\varepsilon p$, proving the equation for $\Psi$.
\end{proof}

\begin{rmk}\normalfont\label{rmk:descreps} Actually, we have that $\varepsilon=s_{p-1}$, where $s_{p-1}$ is as in Proposition \ref{pro:bimodp}. Indeed, the element $\psi=-\sum_{m=1}^{p-1}\chi^{-1}(m)\sigma^m$ comes from applying Proposition \ref{pro:rankpsummary}, and the element $s\in\mathbb{Z}_p$ such that $\psi^p=s\psi$ is the element $s_p$ in Proposition \ref{pro:bimodp}, namely $s=s_p=ps_{p-1}$.
\end{rmk}

The proof for Theorem \ref{thm:hopfalgdimp} remains valid if $r=1$, which corresponds to the case in which $L/K$ is cyclic of degree $p$ and its only Hopf-Galois structure is $H=K[G]$. In that case, we know that $f=\sigma-1_G$ is a generator for $K[G]$ as a $K$-algebra, which is the one typically used to determine the $\mathfrak{A}_{L/K}$-module structure of $\mathcal{O}_L$. However, the element $\Psi$ from Theorem \ref{thm:hopfalgdimp} is another generator; namely $\Psi=-\frac{1}{\alpha}\sum_{m=1}^{p-1}\chi(m)^{-1}\sigma^m$.

A similar situation occurs for $r=2$, that corresponds to the case in which $\widetilde{L}/K$ is dihedral of degree $2p$. In this case, it follows from \cite[Section 3.1]{gildegpdihedral} that $H$ is generated as a $K$-algebra by $w\coloneqq z(\sigma-\sigma^{-1})$, where $z\in M$ is any element such that $z\notin K$ and $z^2\in K$, and in that reference, this is the generator that is used to characterize the $\mathfrak{A}_{L/K}$-freeness of $\mathcal{O}_L$. However, the element $\Psi$ from Theorem \ref{thm:hopfalgdimp} is not of this form. 

In both cases, the advantage of using the generator $\Psi$ is that it satisfies a much simpler degree $p$ polynomial; namely the one given at Corollary \ref{coro:eqdegp}, and this eventually will simplify some computations with respect to the cases $r=1$ and $r=2$. Let us review the big picture for those.

For the Galois case, the equation satisfied by $f=\sigma-\mathrm{Id}$ is determined easily from the relation $\sigma^p=1$ as \begin{equation}\label{eq:galeq}
    f^p=-\sum_{j=1}^{p-1}\binom{p}{j}f^j
\end{equation} (see for instance \cite[Chapitre II, Proof of Lemme 2]{fertonthesis} or \cite[Proposition 6.8 (2)]{delcorsoferrilombardo}). As for the dihedral case, the generator $w=z(\sigma-\sigma^{-1})$ of $H$ satisfies \begin{equation}\label{eq:diheq}
    w^p+c_1z^2w^{p-2}+c_2z^4w^{p-4}+\dots+c_{\frac{p-1}{2}}z^{p-1}w=0,
\end{equation} where $$c_j=\frac{p}{j}\binom{p-j-1}{j-1}=\binom{p-j}{j}+\binom{p-j-1}{j-1}.$$ A proof can be found in \cite[Proposition 3.2]{gildegpdihedral}.

Define $l$ to be $t$ if $r=1$ and $\frac{p+t}{2}$ if $r=2$; define also $g$ to be $f$ if $r=1$ and $w$ if $r=2$. The $\mathfrak{A}_H$-freeness of $\mathcal{O}_L$ was proved by using some of the characterizations provided in Proposition \ref{pro:hnormalfreeness}, where the $\mathcal{O}_K$-lattice $\mathfrak{A}_{\theta}$ is considered for $\theta=\pi_L^a$, where $a$ is the remainder of $l$ mod $p$. An $\mathcal{O}_K$-basis of $\mathfrak{A}_{\theta}$ was determined as $\{\pi_K^{-\nu_i}g^i\}_{i=0}^{p-1}$, where $\nu_i=\Big\lfloor\frac{a+il}{p}\Big\rfloor$, and an $\mathcal{O}_K$-basis of $\mathfrak{A}_{L/K}$ was $\{\pi_K^{-n_i}g^i\}_{i=0}^{p-1}$, where $n_i=\mathrm{min}_{0\leq j\leq p-1-i}(\nu_{i+j}-\nu_j)$. In order to apply Proposition \ref{pro:hnormalfreeness}, expressions of the products of the basis elements with respect to the bases at both $\mathfrak{A}_{\theta}$ and $\mathfrak{A}_{L/K}$ are required; for example to characterize the equality $\mathfrak{A}_{L/K}=\mathfrak{A}_{\theta}$ or to study $\mathfrak{A}_{\theta}$ as a fractional $\mathfrak{A}_{L/K}$-ideal. If the result of a multiplication yields a power of $\Psi$ greater than $p-1$, we use the equation \eqref{eq:galeq} if $r=1$ and \eqref{eq:diheq} if $r=2$ to find such expressions.

In Sections \ref{sec:freescaffold} and \ref{sec:freecontfrac}, we will follow the same ideas for the general case but using the generator $\Psi$ and the equation $\Psi^p=\lambda\Psi$ provided by Corollary \ref{coro:eqdegp}. Our approach will be valid for the already mentioned cases $r=1$ and $r=2$, providing an alternative proof for those.

\subsubsection{The action on the extension}\label{sec:action}

In \cite[Paragraph after Theorem 3.5]{elder2018}, Elder constructed a scaffold for any typical degree $p$ extension, which we shall specify for our extension $L/K$. We will see that Definition \ref{def:scaffold} \ref{def:scaffold2} is fulfilled for the element $\Psi$ from Theorem \ref{thm:hopfalgdimp}. Since $H=K[\Psi]$, this is enough to describe the whole action of $H$ on $L$. However, for our purposes, it will be enough to describe the $L$-valuations of the elements $\Psi\cdot\lambda_n$, where the elements $\lambda_n$, $n\in\mathbb{Z}$, are as in Definition \ref{def:scaffold} \ref{def:scaffold1}.

Let us adopt the notation from Section \ref{sec:scaffoldsp}. Note that since $\mathfrak{a}=(r\circ(-\mathfrak{b}))^{-1}$, for each $n\in\mathbb{Z}$, there is $f_n\in\mathbb{Z}$ such that $n=-\mathfrak{a}(n)b+f_np$ (here we identify $\mathfrak{a}(n)=\mathfrak{a}(r(n))$ for $n\in\mathbb{Z}$).


\begin{thm}[\cite{elder2018}, Theorem 3.5]\label{thm:actiononxi} Let $L/K$ be a typical degree $p$ extension of $p$-adic fields and let $x$ be its primitive element as described in Theorem \ref{thm:eldertypicalp}. Let $H=K[\Psi]$ be its Hopf-Galois structure, where $\Psi$ is as in Theorem \ref{thm:hopfalgdimp}, and let $$\mathfrak{c}\coloneqq pe-(p-1)\ell.$$ Then, $\Psi\cdot1=0$ and, for a given $1\leq i\leq p-1$, we have $$\Psi\cdot x^i\equiv ix^{i-1}\,(\mathrm{mod}\,x^{i-1}\mathfrak{p}_L^{\mathfrak{c}}).$$
\end{thm}

Note that $0<\ell<\frac{pe}{p-1}$ implies that $\mathfrak{c}>0$. Theorem \ref{thm:actiononxi} gives rise to a scaffold on $L/K$ as follows.

\begin{coro}\label{coro:scaffoldpintr} Let $L/K$ be a typical degree $p$ extension of $p$-adic fields and adopt the notation from Theorem \ref{thm:eldertypicalp}. There is a scaffold on $L/K$ that consists in the following data:
\begin{itemize}
    \item The integer numbers $\{\lambda_n\}_{n\in\mathbb{Z}}$ defined by $\lambda_n=x^{\mathfrak{a}(n)}\pi_K^{f_n}$.
    \item The element $\Psi=-\frac{1}{y}\sum_{m=1}^{p-1}\chi(m)^{-1}\sigma^m\in H$ from Theorem \ref{thm:hopfalgdimp}.
\end{itemize}

The shift parameter of the scaffold is $b$ and its precision is $\mathfrak{c}$.
\end{coro}
\begin{proof}
It is a routine verification that Definition \ref{def:scaffold} is fulfilled for the above data. Using Remark \ref{rmk:valx}, $$v_L(\lambda_n)=\mathfrak{a}(n)v_L(x)+pf_n=-\mathfrak{a}(n)b+pf_n=n.$$ On the other hand, if $n_1\equiv n_2\,(\mathrm{mod}\,p)$, then $\mathfrak{a}(n_1)=\mathfrak{a}(n_2)$, so $\frac{\lambda_{n_1}}{\lambda_{n_2}}=\pi_K^{f_{n_1}-f_{n_2}}\in K$. Therefore \ref{def:scaffold1} is satisfied.

Let us check that \ref{def:scaffold2} holds. We note that $\mathfrak{a}(n)=0$ if and only if $p\mid n$. First, let us assume that $p\nmid n$. From the identity $n=-\mathfrak{a}(n)b+pf_n$ we obtain $n+b=-(\mathfrak{a}(n)-1)b+pf_n$, whence $\mathfrak{a}(n+b)=\mathfrak{a}(n)-1$ and $f_n=f_{n+b}$. Now, from Theorem \ref{thm:actiononxi}, \begin{equation*}
    \begin{split}
        \Psi\cdot\lambda_n&\equiv\mathfrak{a}(n)x^{\mathfrak{a}(n)-1}\pi_K^{f_n}\,(\mathrm{mod}\,x^{\mathfrak{a}(n)-1}\pi_K^{pf_n}\mathfrak{p}_L^{\mathfrak{c}})
        \\&\equiv\mathfrak{a}(n)x^{\mathfrak{a}(n+b)}\pi_K^{f_{n+b}}\,(\mathrm{mod}\,x^{\mathfrak{a}(n+b)}\pi_K^{pf_{n+b}}\mathfrak{p}_L^{\mathfrak{c}})
        \\&\equiv\mathfrak{a}(n)\lambda_{n+b}\,(\mathrm{mod}\,\lambda_{n+b}\mathfrak{p}_L^{\mathfrak{c}}),
    \end{split}
\end{equation*} and $u_n\coloneqq\mathfrak{a}(n)$ works as unit because $p\nmid n$. Now suppose that $p\mid n$. Then $\lambda_n\in K$ because $\mathfrak{a}(n)=0$, so $\Psi\cdot\lambda_n=\lambda_n\Psi\cdot1=0$, whence $\Psi\cdot\lambda_n\equiv0\,(\mathrm{mod}\,x^{\mathfrak{a}(n+b)}\mathfrak{p}_L^{\mathfrak{c}})$.
\end{proof}

Thus, the action of $\Psi$ on the elements $\lambda_n$ is described from Remark \ref{rmk:raisesval}: it raises valuations by at least $b$, and exactly this value if and only if $p\nmid n$.


\section{Reduction to the almost maximally ramified case}\label{sec:freescaffold}

Let $L/K$ be a typical degree $p$ extension of $p$-adic fields with Hopf-Galois structure $H=K[\Psi]$. Let $\ell$ be the ramification jump of $L/K$ and let $a$ be its residue class mod $p$. In this section we shall prove Theorem \ref{maintheorem} \ref{mainthm2}-\ref{mainthm3}. Namely, we shall find explicitly an $\mathcal{O}_K$-basis of $\mathfrak{A}_{L/K}$ and characterize the $\mathfrak{A}_{L/K}$-freeness of $\mathcal{O}_L$ when $\ell<\frac{pe}{p-1}-1$, that is, $L/K$ is not almost maximally ramified. Moreover, we shall prove that $a\mid p-1$ is a sufficient condition for freeness, even if $L/K$ is almost maximally ramified.

The main tool will be the scaffold on $L/K$ introduced at Section \ref{sec:action}, and more concretely the specialization of Theorem \ref{thm:criteriascaffold}, from which \ref{mainthm2} will follow directly. Consider the numbers $d(i)$, $w(i)$ from \eqref{eq:paramscaffold} for this scaffold. In Section \ref{sec:freescaffold1}, we will obtain that $d(i)=w(i)$ for all $0\leq i\leq p-1$ is a sufficient condition for freeness, and necessary if $\ell<\frac{pe}{p-1}-1$. In Section \ref{sec:freescaffold2}, we shall see that $d(i)=w(i)$ for all $0\leq i\leq p-1$ if and only if $a\mid p-1$, establishing \ref{mainthm3}. This equivalence does not require that $\ell<\frac{pe}{p-1}-1$. In fact, in Section \ref{sec:freecontfrac}, it will play a role so as to prove Theorem \ref{maintheorem} \ref{mainthm4}.


\subsection{Specifying the criterion for freeness}\label{sec:freescaffold1}

We shall apply Theorem \ref{thm:criteriascaffold} to this situation. From Theorem \ref{thm:eldertypicalp}, we have that $\ell=b+\frac{pc}{r}$, so $\ell\equiv b\,(\mathrm{mod}\,p)$. Then, $a$ is also the remainder of $b$ mod $p$. Now, reducing the equality $a=-\mathfrak{a}(a)b+f_ap$ mod $p$ we deduce that $\mathfrak{a}(a)=p-1$.

The numbers at \eqref{eq:paramscaffold} become $$d(i)=\Big\lfloor\frac{a+ib}{p}\Big\rfloor,$$ $$w(i)=\mathrm{min}_{0\leq j\leq p-1}(d(i+j)-d(i))$$ for all $0\leq i\leq p-1$. It is immediate that $w(i)\leq d(i)$ for all $0\leq i\leq p-1$. Let us write $b=pa_0+a$ for the Euclidean division of $b$ by $p$. Then, \begin{equation}\label{eq:rewritedi}
    d(i)=ia_0+\Big\lfloor(i+1)\frac{a}{p}\Big\rfloor
\end{equation} for all $0\leq i\leq p-1$. In particular, $$d(p-1)=\Big\lfloor\frac{a+(p-1)b}{p}\Big\rfloor=\Big\lfloor\frac{a+(p-1)(pa_0+a)}{p}\Big\rfloor=a+(p-1)a_0.$$

We shall specialize Theorem \ref{thm:criteriascaffold} to this situation. The weak and strong conditions therein are $\mathfrak{c}\geq a$ and $\mathfrak{c}\geq p+a$, respectively. In order to characterize whether they are satisfied, we will find an alternative expression for the precision $\mathfrak{c}$. 

\begin{pro}\label{pro:precscaffold} Let $L/K$ be a typical degree $p$ extension of $p$-adic fields. The precision of the scaffold on $L/K$ can be written as $$\mathfrak{c}=p\Big(e-\frac{p-1}{r}c-d(p-1)\Big)+a.$$ In particular, the remainder of $\mathfrak{c}$ mod $p$ is $a$.
\end{pro}
\begin{proof}
It is a direct calculation:
\begin{equation*}
    \begin{split}
        \mathfrak{c}&=pe-(p-1)\ell
            \\&=pe-(p-1)\Big(b+\frac{pc}{r}\Big)
            \\&=pe-p\frac{p-1}{r}c-(p-1)b
            \\&=pe-p\frac{p-1}{r}c-(p-1)a-p(p-1)a_0
            \\&=pe-p\frac{p-1}{r}c-pa-p(p-1)a_0+a
            \\&=pe-p\frac{p-1}{r}c-p(a+(p-1)a_0)+a
            \\&=p\Big(e-\frac{p-1}{r}c-d(p-1)\Big)+a.
    \end{split}
\end{equation*}
\end{proof}

\begin{rmk}\normalfont From Corollary \ref{coro:eqdegp} we see that $e-\frac{p-1}{r}c$ is exactly the $K$-valuation of the element $\lambda\in\mathcal{O}_K$ such that $\Psi^p=\lambda\Psi$. Hence, we can rewrite the precision of the scaffold in the more compact form $\mathfrak{c}=p(v_K(\lambda)-d(p-1))$.
\end{rmk}

Using this expression, we can characterize easily the weak and strong conditions. 

\begin{pro}\label{pro:condscaffoldp} The weak condition $\mathfrak{c}\geq a$ always holds, while the strong one, $\mathfrak{c}\geq p+a$, does if and only if $\ell<\frac{pe}{p-1}-1$.
\end{pro}
\begin{proof}
    Since $\mathfrak{c}>0$ and $0<a<p$, from the expression of $\mathfrak{c}$ at Proposition \ref{pro:precscaffold} we see immediately that $\mathfrak{c}\geq a$. As for the strong condition, if $\mathfrak{c}\geq p+a$, then $\ell\leq\frac{pe}{p-1}-\frac{p+a}{p-1}<\frac{pe}{p-1}-1$. Conversely, if $\ell<\frac{pe}{p-1}-1$, then $\mathfrak{c}=pe-(p-1)\ell>p-1$, that is, $\mathfrak{c}\geq p$. Since $\mathfrak{c}\equiv a\,(\mathrm{mod}\,p)$, necessarily $\mathfrak{c}\geq p+a$, as we wanted.
\end{proof}

By applying Theorem \ref{thm:criteriascaffold} we obtain the following.

\begin{coro}\label{coro:scaffoldp} Let $L/K$ be a typical degree $p$ extension of $p$-adic fields with ramification jump $\ell$.
\begin{enumerate}[label=\arabic*)]
    \item\label{coro:scaffoldp1} $\{\pi_K^{-w(i)}\Psi^i\}_{i=0}^{p-1}$ is an $\mathcal{O}_K$-basis of $\mathfrak{A}_{L/K}$.
    \item\label{coro:scaffoldp2} If $d(i)=w(i)$ for all $0\leq i\leq p-1$, then $\mathcal{O}_L$ is $\mathfrak{A}_{L/K}$-free. Moreover, if $\ell<\frac{pe}{p-1}-1$, then the converse holds.
\end{enumerate}
\end{coro}
\begin{proof}
Since the weak condition is always satisfied, $\{\pi_K^{-w(i)}\Psi^i\}_{i=0}^{p-1}$ is an $\mathcal{O}_K$-basis of $\mathfrak{A}_{L/K}$ (proving \ref{coro:scaffoldp1}), and if $d(i)=w(i)$ for all $0\leq i\leq p-1$, then $\mathcal{O}_L$ is $\mathfrak{A}_{L/K}$-free. If in addition $\ell<\frac{pe}{p-1}-1$, then the strong condition is satisfied, so $\mathcal{O}_L$ is $\mathfrak{A}_{L/K}$-free if and only if $d(i)=w(i)$ for all $0\leq i\leq p-1$.
\end{proof}

\begin{rmk}\normalfont With the use of the same scaffold, Elder showed that $d(i)=w(i)$ for all $0\leq i\leq p-1$ is a characterization for freeness if $\ell<\frac{pe}{p-1}-2$ (see \cite[Corollary 3.6]{elder2018}). Therefore, Corollary \ref{coro:scaffoldp} \ref{coro:scaffoldp2} is a more general result than his. This is because he imposes the even stronger inequality \begin{equation}\label{eq:weakerineq}
    \mathfrak{c}\geq 2p-1,
\end{equation} given that he studies not only the freeness of $\mathcal{O}_L$, but of all fractional ideals $\mathfrak{p}_L^n$, and following \cite[Remark 3.2]{byottchildselder}, \eqref{eq:weakerineq} allows to apply \cite[Theorem 3.1]{byottchildselder} in this more general situation. If we specify \ref{coro:scaffoldp} to the case in which $L/K$ is Galois, we recover Theorem \ref{thm:cyclicpfreeness} \ref{thm:cyclicpfreeness3}. This shows that, for cyclic degree $p$ extensions of $p$-adic fields $L/K$, the methods from \cite{bertrandiasbertrandiasferton,bertrandiasferton} to study the Galois module structure of $\mathcal{O}_L$ are actually equivalent to the scaffold techniques from \cite{byottchildselder}.
\end{rmk}

\subsection{Characterization of the condition for freeness}\label{sec:freescaffold2}

Looking at Corollary \ref{coro:scaffoldp}, in order to prove Theorem \ref{maintheorem} \ref{mainthm3}, it is enough to show that $d(i)=w(i)$ for all $0\leq i\leq p-1$ if and only if $a\mid p-1$. Now, it is known that the differences $w(i)-d(i)$, $0\leq i\leq p-1$, are completely determined by the shift parameter $b$. Then, any statement involving a Galois scaffold with shift parameter $b$ for a Galois degree $p$ extension also holds for our situation (see \cite[Proof of Corollary 3.6]{elder2018}). Therefore, the equivalence above is validated and Theorem \ref{maintheorem} \ref{mainthm3} is proved. For expository purposes, we shall give a proof by using Proposition \ref{pro:hnormalfreeness}. Our proof will make use of the generator $\Psi$ and the equation $\Psi^p=\lambda\Psi$. Consequently, we obtain an alternative and simpler proof than the known ones for the Galois case \cite[Section 2]{ferton1972} and the dihedral case \cite[Section 6.3]{gildegpdihedral}. Note that we do not assume that $\ell<\frac{pe}{p-1}-1$, so the results in this section can be used also in the almost maximally ramified case.

Recall that $\mathfrak{a}(a)=p-1$. Let $\theta=\pi_L^a$. Since $\lambda_a=x^{p-1}\pi_K^{pf_a}$ and $v_L(\lambda_a)=a$, adjusting $\pi_L$ and $\pi_K$ suitably, we can assume without loss of generality that $\lambda_a=\theta$. By Proposition \ref{pro:anormbasisscaffold}, $\theta$ is an $H$-normal basis generator for $L$.  Since $p\nmid a$, from Corollary \ref{coro:scaffoldpintr} we have that $v_L(\Psi\cdot\theta)=a+b$. After an easy induction we obtain:

\begin{lema}\label{lem:valacttheta} Let $L/K$ be a typical degree $p$ extension of $p$-adic fields with ramification jump $\ell=b+\frac{pc}{r}$ and let $a$ be the residue class of $\ell$ mod $p$. For all $1\leq i\leq p-1$, we have $$v_L(\Psi^i\cdot\theta)=a+ib.$$
\end{lema}

For later use, let us write $$\frac{b}{p}=[a_0;a_1,\dots,a_n]$$ for the continued fraction expansion of $\frac{b}{p}$ and $\{q_i\}_{i=0}^n$ for the finite sequence of convergents as defined in Section \ref{sec:contfrac}.

\begin{pro}\label{pro:intbasis} Let $\theta=\pi_L^a$. The $\mathcal{O}_K$-module $\mathfrak{A}_{\theta}$ is free with $\mathcal{O}_K$-basis $$\{\pi_K^{-d(i)}\Psi^i\}_{i=0}^{p-1}.$$ Consequently, $\{\pi_K^{-d(i)}\Psi^i\cdot\theta\}_{i=0}^{p-1}$ is an integral basis for $L/K$.
\end{pro}
\begin{proof}
For $h\in H$, let us write $h=\sum_{i=0}^{p-1}h_i\Psi^i$ with $h_i\in K$. We have that $h\cdot\theta\in\mathcal{O}_L$ if and only if \begin{equation}\label{eq:comlinint}
    \sum_{i=0}^{p-1}h_i\Psi^i\cdot\theta\in\mathcal{O}_L.
\end{equation} We have $$v_L(h_i\Psi^i\cdot\theta)=pv_K(h_i)+a+ib\equiv(i+1)a\,(\mathrm{mod}\,p)$$ for every $0\leq i\leq p-1$, which are all different mod $p$. Then \eqref{eq:comlinint} is equivalent to $h_i\Psi^i\cdot\theta\in\mathcal{O}_L$ for each $0\leq i\leq p-1$. Now, $v_L(h_i\Psi^i\cdot\theta)\geq0$ if and only if $v_K(h_i)\geq-\frac{a+ib}{p}$, which is equivalent to $v_K(h_i)\geq-d(i)$ for all $0\leq i\leq p-1$. This finishes the proof.
\end{proof}

This result together with Corollary \ref{coro:scaffoldp} \ref{coro:scaffoldp1} allows us to deduce that $d(i)=w(i)$ for all $0\leq i\leq p-1$ if and only if $\mathfrak{A}_{L/K}=\mathfrak{A}_{\theta}$. By Proposition \ref{pro:hnormalfreeness} and the paragraph afterwards, this happens if and only if $\mathfrak{A}_{\theta}$ is a ring. This is the condition we will work with.

\begin{lema}\label{lem:subr} Let $\theta=\pi_L^a$ and $1\leq m\leq p-1$. Given $d\in\mathbb{Z}$, $\pi_K^{-d}\Psi^{p-1+m}\in\mathfrak{A}_{\theta}$ if and only if  $d\leq e-\frac{p-1}{r}c+d(m)$.
\end{lema}
\begin{proof}
Since $p\leq p-1+m\leq 2p-1$, we can describe $\pi_K^{-d}\Psi^{p-1+m}$ by using Corollary \ref{coro:eqdegp}: $$\pi_K^{-d}\Psi^{p-1+m}=\pi_K^{-d}\Psi^{m-1}\Psi^p=\pi_K^{-d}\Psi^{m-1}\lambda\Psi=\lambda\pi_K^{-d}\Psi^m$$ Now, $$v_L(\pi_K^{-d}\Psi^{p-1+m}\cdot\theta)=p\Big(e-\frac{p-1}{r}c-d\Big)+a+mb,$$ and this is non-negative if and only if $$d\leq e-\frac{p-1}{r}c+d(m).$$
\end{proof}

\begin{pro}\label{pro:twoimpl} The $\mathcal{O}_K$-module $\mathfrak{A}_{\theta}$ is a ring if and only if for every pair of integers $0\leq i,j\leq p-1$ the following implications are satisfied:
\begin{enumerate}[label=\arabic*)]
    \item\label{pro:twoimpl1} If $i+j\leq p-1$, then $d(i)+d(j)\leq d(i+j)$.
    \item\label{pro:twoimpl2} If $i+j\geq p$, then $d(i)+d(j)\leq e-\frac{p-1}{r}c+d(i+j+1-p)$.
\end{enumerate}
\end{pro}
\begin{proof}
Since $\mathfrak{A}_{\theta}$ is a full $\mathcal{O}_K$-lattice of $H$ and contains the identity, it is a ring if and only if it is closed under the multiplication of $H$. Moreover, this is equivalent to the product being closed for the elements of the $\mathcal{O}_K$-basis $\{\pi_K^{-d(k)}\Psi^k\}_{k=0}^{p-1}$ for $\mathfrak{A}_{\theta}$. Hence, it is immediate that if $\mathfrak{A}_{\theta}$ is a ring, then \ref{pro:twoimpl1} holds. As for the case $i+j\geq p$, we call $m=i+j-(p-1)$, so that $i+j=p-1+m$. Now, $$\pi_K^{-d(i)}\Psi^i\pi_K^{-d(j)}\Psi^j=\pi_K^{-(d(i)+d(j))}\Psi^{p-1+m}.$$ By Lemma \ref{lem:subr}, this belongs to $\mathfrak{A}_{\theta}$ if and only if $d(i)+d(j)\leq e-\frac{p-1}{r}c+d(i+j+1-p)$, which gives the necessity of \ref{pro:twoimpl2}. On the other hand, it is straightforward to prove that if \ref{pro:twoimpl1} and \ref{pro:twoimpl2} are valid, then $\mathfrak{A}_{\theta}$ is a ring.
\end{proof}

\begin{rmk}\normalfont Lemma \ref{lem:subr} and Proposition \ref{pro:twoimpl} illustrate how the simplicity of the equation $\Psi^p=\lambda\Psi$ from Corollary \ref{coro:eqdegp} translates into the simplicity of the calculations required to study the $\mathfrak{A}_H$-freeness of $\mathcal{O}_L$. For instance, the analogue of Lemma \ref{lem:subr} in the dihedral case is \cite[Lemma 6.9]{gildegpdihedral}, whose proof makes use of the equation \eqref{eq:diheq} satisfied by $w=z(\sigma-\sigma^{-1})$. The reader can compare it with the proof of Lemma \ref{lem:subr} to see the advantage of using the generator $\Psi$.
\end{rmk}

We obtain immediately that $\mathfrak{A}_{\theta}=\mathfrak{A}_{L/K}$ if and only if the two implications above hold. We shall prove that this happens if and only if $a\mid p-1$, obtaining the desired characterization.

\begin{pro}\label{pro:charactequality} We have that $\mathfrak{A}_{\theta}=\mathfrak{A}_{L/K}$ if and only if $a\mid p-1$.
\end{pro}
\begin{proof}
We shall prove that Proposition \ref{pro:twoimpl} \ref{pro:twoimpl1} \ref{pro:twoimpl2} hold if and only if $a\mid p-1$.

Let $0\leq i,j\leq p-1$ be such that $i+j\geq p$. We claim that the inequality $d(i)+d(j)\leq e-\frac{p-1}{r}c+d(i+j+1-p)$ always holds. Indeed, it does if and only if $$\Big\lfloor(i+1)\frac{a}{p}\Big\rfloor+\Big\lfloor(j+1)\frac{a}{p}\Big\rfloor\leq e-\frac{p-1}{r}c-(p-1)a_0+\Big\lfloor(i+j+2-p)\frac{a}{p}\Big\rfloor.$$ Now, we have that $$\Big\lfloor(i+1)\frac{a}{p}\Big\rfloor+\Big\lfloor(j+1)\frac{a}{p}\Big\rfloor\leq\Big\lfloor(i+j+2)\frac{a}{p}\Big\rfloor=a+\Big\lfloor(i+j+2-p)\frac{a}{p}\Big\rfloor.$$ Hence, we obtain \begin{equation*}
    \begin{split}
        d(i)+d(j)&\leq(i+j)a_0+a+\Big\lfloor(i+j+2-p)\frac{a}{p}\Big\rfloor
        \\&=a+(p-1)a_0+(i+j+1-p)a_0+\Big\lfloor(i+j+2-p)\frac{a}{p}\Big\rfloor
        \\&=d(p-1)+d(i+j+1-p).
    \end{split}
\end{equation*}

Recall that from Proposition \ref{pro:precscaffold}, we have $d(p-1)\leq e-\frac{p-1}{r}c$. Then our claim follows immediately.

Let $0\leq i,j\leq p-1$ be such that $i+j\leq p-1$. We shall show that $a\mid p-1$ if and only if $d(i)+d(j)\leq d(i+j)$. This inequality is equivalent to $$\Big\lfloor(i+1)\frac{a}{p}\Big\rfloor+\Big\lfloor(j+1)\frac{a}{p}\Big\rfloor\leq\Big\lfloor(i+j+1)\frac{a}{p}\Big\rfloor.$$

Let us assume that $a\mid p-1$ and call $d=\frac{p-1}{a}$. By \cite[Lemma 7.4]{delcorsoferrilombardo}, we have that $\Big\lfloor(k+1)\frac{a}{p}\Big\rfloor=\Big\lfloor\frac{k}{d}\Big\rfloor$ for every $1\leq k\leq p-1$. Then, $$\Big\lfloor(i+1)\frac{a}{p}\Big\rfloor+\Big\lfloor(j+1)\frac{a}{p}\Big\rfloor=\Big\lfloor\frac{i}{d}\Big\rfloor+\Big\lfloor\frac{j}{d}\Big\rfloor\leq\Big\lfloor\frac{i+j}{d}\Big\rfloor=\Big\lfloor(i+j+1)\frac{a}{p}\Big\rfloor.$$

Conversely, let us assume that $a\nmid p-1$ and prove that there are integers $0\leq i,j\leq p-1$ with $i+j>p-1$ such that \begin{equation*}
\Big\lfloor(i+1)\frac{a}{p}\Big\rfloor+\Big\lfloor(j+1)\frac{a}{p}\Big\rfloor>\Big\lfloor(i+j+1)\frac{a}{p}\Big\rfloor.
\end{equation*} Equivalently, \begin{equation}\label{eq:inequalityfractij}
    \widehat{(i+1)\frac{a}{p}}+\widehat{(j+1)\frac{a}{p}}<\frac{a}{p}+\widehat{(i+j+1)\frac{a}{p}}.
\end{equation}

Suppose that $n$ is odd, and choose $$i=q_{n-1}-1,\quad j=p-2q_{n-1}.$$ First, let us check that $j\geq0$. Recall that $q_n=a_nq_{n-1}+a_{n-2}$ and that $q_n=p$. Moreover, by definition of continued fraction expansion, we must have $a_{n-2}>0$ and $a_n>1$. Then $p\geq a_nq_{n-1}\geq 2q_{n-1}$, as we wanted. From Proposition \ref{pro:contfrac} \ref{pro:contfrac4} we have $\widehat{(i+1)\frac{a}{p}}=\widehat{q_{n-1}\frac{a}{p}}=\frac{1}{p}<\frac{a}{p}$ (note that $a>1$ because we are under the assumption that $a\nmid p-1$). Hence, in order to obtain \eqref{eq:inequalityfractij}, it is enough to prove that $\widehat{(j+1)\frac{a}{p}}\leq\widehat{(i+j+1)\frac{a}{p}}$. Indeed, since $i+j+1=p-q_{n-1}=p-(i+1)$, we have $$\widehat{(i+j+1)\frac{a}{p}}=1-\widehat{(i+1)\frac{a}{p}}=\frac{p-1}{p},$$ which is an upper bound for $\widehat{(j+1)\frac{a}{p}}$.

Now, assume that $n$ is even, and choose $$i=q_{n-2}-1,\quad j=q_{n-1}-q_{n-2}.$$ Note that since $a\nmid p-1$ and $n$ is even, we have that $n\geq4$, so that $n-2\geq2$. By Proposition \ref{pro:contfrac} \ref{pro:contfrac1} \ref{pro:contfrac3}, $\widehat{(i+1)\frac{a}{p}}=\widehat{q_{n-2}\frac{a}{p}}\leq\frac{a}{p}$. On the other hand, in this case we have $i+j+1=q_{n-1}$, whence $\widehat{(i+j+1)\frac{a}{p}}=\widehat{q_{n-1}\frac{a}{p}}=\frac{p-1}{p}$, and the same argument as in the previous case proves the validity of \eqref{eq:inequalityfractij}.
\end{proof}

As already highlighted, we are not assuming that $\ell<\frac{pe}{p-1}-1$, so the equivalence $\mathfrak{A}_{L/K}=\mathfrak{A}_{\theta}$ with $\theta=\pi_L^a$ if and only if $a\mid p-1$ is established under the only assumption that $L/K$ is typical. Consequently, if $a\nmid p-1$, then $\mathfrak{A}_{L/K}\neq\mathfrak{A}_{\theta}$, but $\mathcal{O}_L$ could still be $\mathfrak{A}_{L/K}$-free. Namely, we may have that $\mathfrak{A}_{\theta}=\mathfrak{A}_{L/K}$ for some other choice of an element $\theta\in\mathcal{O}_L$ that generates an $H$-normal basis for $L$. On the other hand, the equivalence above is consistent with Theorem \ref{maintheorem} \ref{mainthm4}, because of the following characterization for $a\mid p-1$ in terms of the continued fraction expansion of $\frac{a}{p}$:

\begin{pro}\label{pro:charactatheta} Assume the notation above and let $n$ be the length of the continued fraction expansion of $\frac{b}{p}$. Then $a\mid p-1$ if and only if $n\leq 2$.
\end{pro}
\begin{proof}
We can assume without loss of generality that $b=a$, since the continued fraction expansions $\frac{a}{p}$ and $\frac{b}{p}$ have the same coefficients except the integer part, and in particular the same length. Suppose that $n\leq 2$. If $n=1$, then $\frac{a}{p}=[0;a_1]=\frac{1}{a_1}$, so $a=1\mid p-1$. Otherwise $n=2$, and we have that $\frac{a}{p}=[0;a_1,a_2]=\frac{a_2}{a_1a_2+1}$ with $a_2$ and $a_1a_2+1$ coprime, so $q_2=a_1a_2+1$. Since $n=2$ we have that $q_2=p$, so necessarily $a_2=a$ and $a\mid p-1$. Conversely, assume that $a\mid p-1$. If $a=1$, then $\frac{a}{p}=\frac{1}{p}=[0;p]$ and $n=1$. Otherwise, since $\frac{p-1}{a}\in\mathbb{Z}$, we can write $$\frac{a}{p}=\cfrac{1}{\frac{p-1}{a}+\cfrac{1}{a}}=\Big[0;\frac{p-1}{a},a\Big]$$ and $n=2$.
\end{proof}

\begin{rmk}\normalfont In fact, we have proved that:
\begin{itemize}
    \item $n=1$ if and only if $a=1$.
    \item $n=2$ if and only if $a\neq1$ and $a\mid p-1$.
\end{itemize}
\end{rmk}

To sum up, we have proved the following:

\begin{pro} Let $L/K$ be a degree $p$ extension of $p$-adic fields with ramification jump $\ell$, let $a$ be the residue class of $\ell$ mod $p$ and let $\theta=\pi_L^a$. The following statements are equivalent:
\begin{enumerate}[label=\arabic*)]
    \item $\mathfrak{A}_{\theta}=\mathfrak{A}_{L/K}$.
    \item $d(i)=w(i)$ for all $0\leq i\leq p-1$.
    \item The length of the continued fraction expansion of $\frac{a}{p}$ is at most $2$.
    \item $a\mid p-1$.
\end{enumerate}
\end{pro}

\section{The almost maximally ramified case}\label{sec:freecontfrac}

This section is devoted to prove Theorem \ref{maintheorem} \ref{mainthm4}. Namely, let $L/K$ be a degree $p$ extension of $p$-adic fields such that $\frac{pe}{p-1}-1\leq\ell\leq\frac{pe}{p-1}$. The case that $\ell=\frac{pe}{p-1}$ corresponds to when $L/K$ is maximally ramified and has been already discussed in Section \ref{sec:maxram}, where we obtained that $\mathcal{O}_L$ is $\mathfrak{A}_{L/K}$-free. Moreover, recall that $\ell<\frac{pe}{p-1}$ if and only if $a\neq0$, that is, $L/K$ is typical. Hence, we can assume that $L/K$ is typical and $$\frac{pe}{p-1}-1\leq\ell<\frac{pe}{p-1}.$$

Recall from Propositions \ref{pro:charactequality} and \ref{pro:charactatheta} that $\mathfrak{A}_{\theta}=\mathfrak{A}_{L/K}$ with $\theta=\pi_L^a$ if and only if $n\leq 2$. This equivalence is useless to prove Theorem \ref{maintheorem} \ref{mainthm4}. We do not proceed, however, to study such an equality for other choices of $H$-normal basis generators $\theta$. Instead, we keep considering $\mathfrak{A}_{\theta}$ with $\theta=\pi_L^a$ and use Proposition \ref{pro:hnormalfreeness} \ref{pro:hnormalfreeness4} and \ref{pro:hnormalfreeness5}, by which $\mathcal{O}_L$ is $\mathfrak{A}_{L/K}$-free if and only if $\mathfrak{A}_{\theta}$ is principal as an $\mathfrak{A}_{L/K}$-fractional ideal for this particular choice of $\theta$.

\subsection{The strategy of proof}

Let $\theta=\pi_L^a$ and for each $\alpha\in\mathfrak{A}_{\theta}$, let us consider the map \begin{equation}\label{eq:psialpha}
    \begin{array}{rccl}
    \psi_{\alpha}\colon & \mathfrak{A}_{L/K} & \longrightarrow & \mathfrak{A}_{\theta}, \\
     & \lambda & \longmapsto & \lambda\alpha.
\end{array}
\end{equation} If we fix $\mathcal{O}_K$-bases of $\mathfrak{A}_{L/K}$ and $\mathfrak{A}_{\theta}$ and call $M(\alpha)$ the matrix of $\psi_{\alpha}$ with respect to these bases, then $\mathfrak{A}_{\theta}$ is $\mathfrak{A}_{L/K}$-principal with generator $\alpha$ if and only if \begin{equation}\label{eq:detMalpha}
    \mathrm{det}(M(\alpha))\not\equiv0\,(\mathrm{mod}\,\mathfrak{p}_K).
\end{equation} If we find some $\alpha\in\mathfrak{A}_{\theta}$ satisfying \eqref{eq:detMalpha}, then $\mathcal{O}_L$ is $\mathfrak{A}_{L/K}$-free. On the contrary, if we prove that there is no $\alpha\in\mathfrak{A}_{\theta}$ with this property, then $\mathcal{O}_L$ is not $\mathfrak{A}_{L/K}$-free.

Let us fix the already known bases: $\{\pi_L^{-w(i)}\Psi^i\}_{i=0}^{p-1}$ for $\mathfrak{A}_{L/K}$ and $\{\pi_L^{-d(i)}\Psi^i\}_{i=0}^{p-1}$ for $\mathfrak{A}_{\theta}$ (see Corollary \ref{coro:scaffoldp} \ref{coro:scaffoldp1} and Proposition \ref{pro:intbasis}, respectively). Let $\alpha\in\mathfrak{A}_{\theta}$, and write $\alpha=\sum_{k=0}^{p-1}x_k\pi_K^{-d(k)}\Psi^k$ with $x_k\in\mathcal{O}_K$. Then, $$M(\alpha)=\sum_{k=0}^{p-1}x_kM(\pi_K^{-d(k)}\Psi^k).$$ Let us denote $M(\pi_K^{-d(k)}\Psi^k)=(\mu_{j,i}^{(k)})_{j,i=0}^{p-1}$ for every $0\leq k\leq p-1$. 

By definition of matrix attached to a linear map, $$\Psi_{\pi_K^{-d(k)}\Psi^k}(\pi_K^{-w(i)}\Psi^i)=\sum_{j=0}^{p-1}\mu_{j,i}^{(k)}\pi_K^{-d(j)}\Psi^j.$$ On the other hand, using the definition of the map $\Psi_{\pi_K^{-d(k)}\Psi^k}$: $$\Psi_{\pi_K^{-d(k)}\Psi^k}(\pi_K^{-w(i)}\Psi^i)=\pi_K^{-d(k)-w(i)}\Psi^{k+i}.$$ Using the uniqueness of coordinates, we see that \begin{equation}\label{eq:coordbasis}
    \mu_{j,i}^{(k)}=C(\Psi^j,\Psi^{k+i})\pi_K^{d(j)-d(k)-w(i)},
\end{equation} where $C(\Psi^j,\Psi^{k+i})$ is the coefficient of $\Psi^j$ in the expression of $\Psi^{k+i}$ with respect to the $K$-basis $\{\Psi^l\}_{l=0}^{p-1}$ of $H$. Its $K$-valuation is determined trivially if $k+i\leq p-1$ and by the equation $\Psi^p=\lambda\Psi$ from Corollary \ref{coro:eqdegp} otherwise. 

We shall determine whether or not the entries $\mu_{j,i}^{(k)}$ vanish mod $\mathfrak{p}_K$. We shall denote such a projection by $\overline{\mu}_{j,i}^{(k)}$. Looking at \eqref{eq:coordbasis}, this amounts to characterize whether $d(j)-d(k)-w(i)$ is equal to $0$ or not. In the case that $L/K$ is Galois, this is done by establishing a link between the numbers $w(i)$, the fractional part of the numbers $h\frac{a}{p}$, where $1\leq h\leq p-1$, and the semiconvergents of the continued fraction expansion of $\frac{b}{p}$ (see \cite[Chapitre II, \textsection 5]{fertonthesis}). In \cite[Section 7]{gildegpdihedral}, this strategy is adapted to the case in which the normal closure of $L/K$ is dihedral of degree $2p$. In what follows, we adapt it to the general case. 

Even if some results and notions at the following sections are rather expository, we include them with complete proofs and explanations for the benefit of the reader.

\subsection{Semiconvergents and fractional parts}

Let us consider the following assignment: for each integer $1\leq h\leq p$, we consider the fractional part $\widehat{h\frac{a}{p}}$, which is a rational number between $0$ and $1$ with denominator (of the fraction in irreducible form) dividing $p$. This map can be visualized graphically by fixing in the real plane a circle $C$ with circumference $1$ and centered at the origin. Let $O$ be the point of $C$ over the $X$-axis in the right half semiplane. The fractional part $\widehat{h\frac{a}{p}}$ is identified with the point $M_h$ in $C$ with the property that $\widehat{h\frac{a}{p}}$ is the length of the arc from $O$ to $M_h$ in anticlockwise direction (see Figure \ref{fig1}). We can formalize the notion of distance between two of these points as follows.

\begin{defi} We define the distance within $C$ between two integers $h,k\in\mathbb{Z}$ as \begin{equation}\label{eq:distanceint}
    d_C(h,k)=\Big|\Big|(h-k)\frac{a}{p}\Big|\Big|=\Big|\widehat{h\frac{a}{p}}-\widehat{k\frac{a}{p}}\Big|
    \end{equation} (see Figure \ref{fig2}). If $\widehat{k\frac{a}{p}}<\widehat{h\frac{a}{p}}$, $d_C(h,k)$ is the length of the arc from $M_k$ to $M_h$ in anticlockwise direction.
\end{defi}

We can also formalize the situation in which a point of $C$ lies between other two.

\begin{defi} We say that a point $M_h$ of $C$ lies between other two points $M_{h'}$, $M_{h''}$ with $\widehat{h'\frac{a}{p}}<\widehat{h''\frac{a}{p}}$ if lies in the arc in anticlockwise direction from $M_{h'}$ to $M_{h''}$, that is, $\widehat{h'\frac{a}{p}}\leq\widehat{h\frac{a}{p}}\leq\widehat{h''\frac{a}{p}}$ (see Figure \ref{fig3}).
\end{defi}

Note that we include in this definition the situation that $M_h$ is either of the points $M_{h'}$ or $M_{h''}$.

\begin{figure}[h!]
\begin{minipage}{0.3 \textwidth}
    \centering
    \begin{tikzpicture}
        \draw[opacity=0.7] (-2,0) -- (2,0);
        \draw[opacity=0.7] (0,-2) -- (0,2);
        \draw (0,0) circle (1.5);

        \draw[line width=0.5mm] (1.5,0) arc (0:60:1.5) node at (30:2) {$\widehat{h\frac{a}{p}}$};
        
        \fill (1.5,0) circle (0.1) node [below right] {$O$};
        \fill (60:1.5) circle (0.1) node [above right] {$M_h$};

    \end{tikzpicture}
    \caption{Point $M_h$ \\ assigned to $h\in\mathbb{Z}$.}
    \label{fig1}
\end{minipage}
\begin{minipage}{0.3 \textwidth}
    \centering
    \begin{tikzpicture}
        \draw[opacity=0.7] (-2,0) -- (2,0);
        \draw[opacity=0.7] (0,-2) -- (0,2);
        \draw (0,0) circle (1.5);

        \draw[line width=0.5mm] (60:1.5) arc (60:160:1.5) node at (110:2) {$d_C(h,k)$};

        \fill (1.5,0) circle (0.1) node [below right] {$O$};
        \fill (60:1.5) circle (0.1) node [above right] {$M_k$};
        \fill (160:1.5) circle (0.1) node [above left] {$M_h$};
    
    \end{tikzpicture}
    \caption{The distance \\ between $M_h$ and $M_k$.}
    \label{fig2}
\end{minipage}
\begin{minipage}{0.3 \textwidth}
    \centering
    \begin{tikzpicture}
        \draw[opacity=0.7] (-2,0) -- (2,0);
        \draw[opacity=0.7] (0,-2) -- (0,2);
        \draw (0,0) circle (1.5);

        \draw[line width=0.5mm] (50:1.5) arc (50:190:1.5);

        \fill (1.5,0) circle (0.1) node [below right] {$O$};
        \fill (50:1.5) circle (0.1) node [above right] {$M_{h'}$};
        \fill (120:1.5) circle (0.1) node [above left] {$M_h$};
        \fill (190:1.5) circle (0.1) node [left] {$M_{h''}$};
        
    \end{tikzpicture}
    \caption{Point $M_h$ \\ between $M_{h'}$ and $M_{h''}$.}
    \label{fig3}
\end{minipage}

\end{figure}

Let us introduce the following set: 
\begin{equation}\label{eq:setE}
E=\Big\{h\in\mathbb{Z}\,\Big|\,1\leq h< p,\,1\leq h'<h\,\Longrightarrow\,\widehat{h'\frac{a}{p}}>\widehat{h\frac{a}{p}}\Big\}.
\end{equation}

The usefulness of the set $E$ for our purposes is illustrated by the following result, that shows a description of the numbers $w(i)$ depending on whether or not $p-i\in E$.

\begin{lema}\label{lem:nuiE} Let $a_0=\lfloor\frac{b}{p}\rfloor$. For each $0\leq i\leq p-1$, we have that $w(i)=ia_0+\Big\lfloor i\frac{a}{p}\Big\rfloor+\epsilon$, where $\epsilon=0$ if $p-i\notin E$ and $\epsilon=1$ if $p-i\in E$.
\end{lema}
\begin{proof}
Looking at \eqref{eq:rewritedi} and the definition of $w(i)$, we see that $$w(i)=ia_0+\mathrm{min}_{0\leq j\leq p-1-i}\Big(\Big\lfloor(i+j+1)\frac{a}{p}\Big\rfloor-\Big\lfloor(j+1)\frac{a}{p}\Big\rfloor\Big).$$ For $i=0$, we have $w(0)=0$ and $\epsilon=0$ because $p\notin E$. Thus, we assume that $i\neq 0$ in the sequel.

Given $0\leq j\leq p-1-i$, \begin{align*}
    (i+j+1)\frac{a}{p}-(j+1)\frac{a}{p}&=\Big\lfloor(i+j+1)\frac{a}{p}\Big\rfloor-\Big\lfloor(j+1)\frac{a}{p}\Big\rfloor+\widehat{(i+j+1)\frac{a}{p}}-\widehat{(j+1)\frac{a}{p}}, \\
    (i+j+1)\frac{a}{p}-(j+1)\frac{a}{p}&=i\frac{a}{p}=\Big\lfloor i\frac{a}{p}\Big\rfloor+\widehat{i\frac{a}{p}}.
\end{align*} Hence, $$\Big\lfloor(i+j+1)\frac{a}{p}\Big\rfloor-\Big\lfloor(j+1)\frac{a}{p}\Big\rfloor=\Big\lfloor i\frac{a}{p}\Big\rfloor+\widehat{i\frac{a}{p}}+\widehat{(j+1)\frac{a}{p}}-\widehat{(i+j+1)\frac{a}{p}}.$$ This is an integer number, so $$\Big\lfloor(i+j+1)\frac{a}{p}\Big\rfloor-\Big\lfloor(j+1)\frac{a}{p}\Big\rfloor=\begin{cases}
\lfloor i\frac{a}{p}\rfloor+1 & \hbox{if }\widehat{i\frac{a}{p}}+\widehat{(j+1)\frac{a}{p}}>1, \\
\lfloor i\frac{a}{p}\rfloor & \hbox{if }\widehat{i\frac{a}{p}}+\widehat{(j+1)\frac{a}{p}}\leq1. \\
\end{cases}$$ Note that $\widehat{i\frac{a}{p}}+\widehat{(j+1)\frac{a}{p}}>1$ if and only if $\widehat{(j+1)\frac{a}{p}}>\widehat{(p-i)\frac{a}{p}}$.

Suppose that $p-i\in E$. Then $\widehat{(j+1)\frac{a}{p}}\geq\widehat{(p-i)\frac{a}{p}}$ for every $0\leq j\leq p-1-i$ (with equality if and only if $j=p-i-1$), and in particular for the value of $j$ such that $\Big\lfloor(i+j+1)\frac{a}{p}\Big\rfloor-\Big\lfloor(j+1)\frac{a}{p}\Big\rfloor$ is minimal. Then $w(i)=ia_0+\lfloor i\frac{a}{p}\rfloor+1$. Otherwise, if $p-i\notin E$, then there is some $0\leq j<p-i-1$ such that $\widehat{(j+1)\frac{a}{p}}\leq\widehat{(p-i)\frac{a}{p}}$. From the previous calculations we deduce that $\Big\lfloor(i+j+1)\frac{a}{p}\Big\rfloor-\Big\lfloor(j+1)\frac{a}{p}\Big\rfloor=\lfloor i\frac{a}{p}\rfloor$, which is a minimum for this value of $j$. Then $w(i)=ia_0+\Big\lfloor i\frac{a}{p}\Big\rfloor$.
\end{proof}

This is stated originally at \cite[Proposition 1]{bertrandiasbertrandiasferton} and proved at \cite[Chapitre II, Proposition 5]{fertonthesis}. The proof is also replicated at \cite[Lemma 7.1]{gildegpdihedral}.

Recall from Section \ref{sec:contfrac} that we refer to the denominators $q$ of the convergents (resp. semiconvergents) $\frac{p}{q}$ of the continued fraction expansion of $\frac{b}{p}$ as convergents (resp. semiconvergents). The link between the map $h\mapsto M_h$ above and the continued fraction expansion of $\frac{b}{p}$ lies in the following result.

\begin{pro}\label{pro:paramE} Write $\frac{b}{p}=[a_0;a_1,\dots,a_n]$ for the continued fraction expansion of $\frac{b}{p}$. The elements of the set $E$ are the even-indexed semiconvergents $$q_{2i,r_{2i}}=r_{2i}q_{2i+1}+q_{2i},\quad0\leq i<\frac{n-1}{2}.$$ Concretely: 
\begin{equation*}
    \begin{split}
        E=\Big\{q_{2i,r_{2i}}\,\Big|\,0\leq i<\frac{n-1}{2},\,r_{2i}\in\mathbb{Z},&\,0\leq r_{2i}<a_{2i+2}\hbox{ if }i\neq\frac{n-3}{2},\\&\,0\leq r_{2i}\leq a_{2i+2}\hbox{ if }i=\frac{n-3}{2}\Big\}.
    \end{split}
\end{equation*}
\end{pro}

In other words, the elements of $E$ are all the even-indexed semiconvergents (including convergents) from $q_0=1$ to $q_{n-1}$ if $n$ is odd and to $q_n$ if $n$ is even. This result is stated at \cite[Lemme 1]{bertrandiasbertrandiasferton} and a complete proof can be found at \cite[Chapitre II, Lemme 4]{fertonthesis} (the reader can also consult \cite[Section 7.1]{gildegpdihedral}). Let us go over this proof and the techniques that are used. 

Prior to the proof of Proposition \ref{pro:paramE}, we shall establish the disposition of the points of $C$ corresponding to the even-indexed convergents and semiconvergents $q_{2i,r}$, $0\leq r\leq a_{2i+2}$.

\begin{pro}\label{pro:arcsemiconv} Let $0\leq i<\frac{n-1}{2}$. The points of $C$ corresponding to the semiconvergents between $q_{2i}$ and $q_{2i+2}$ lie in the arc between $M_{q_{2i+2}}$ and $M_{q_{2i}}$, in such a way that for each $1\leq r\leq a_{2i+2}$, $M_{q_{2i,r-1}}$ and $M_{q_{2i,r}}$ are consecutive points with distance $||q_{2i+1}\frac{a}{p}||$ from each other.
\end{pro}
\begin{figure}[h!]
    \centering
    \begin{tikzpicture}
        \draw (30:7) arc (30:90:7);
        
        \draw[line width=0.5mm] (30:7) arc (30:42:7);
        \draw[line width=0.5mm] (66:7) arc (66:78:7);
        \draw[line width=0.5mm] (78:7) arc (78:90:7);
        
        \fill (30:7) circle (0.1) node [above right] {$M_{q_{2i+2}}$};
        \fill (42:6.3) node [below] {$||q_{2i+1}\frac{a}{p}||$};
        \fill (42:7) circle (0.1) node [above right] {$M_{q_{2i,a_{2i+2}-1}}$};
        \fill (66:7) circle (0.1) node [above right] {$M_{q_{2i,2}}$};
        \fill (76:6.7) node [below] {$||q_{2i+1}\frac{a}{p}||$};
        \fill (78:7) circle (0.1) node [above right] {$M_{q_{2i,1}}$};
        \fill (90:6.7) node [below] {$||q_{2i+1}\frac{a}{p}||$};
        \fill (90:7) circle (0.1) node [above right] {$M_{q_{2i}}$};
        
    \end{tikzpicture}
    \caption{Arc within the circle $C$ between points $M_{q_{2i}}$, $M_{q_{2i+2}}$.}
    \label{fig5}
\end{figure}
\begin{proof}
We consider the statements at Proposition \ref{pro:contfrac}. Since $q_{2i}$ and $q_{2i+2}$ have even subscripts, by \ref{pro:contfrac3}, we have that $\widehat{q_{2i}\frac{a}{p}}$ and $\widehat{q_{2i+2}\frac{a}{p}}$ are lower than $\frac{1}{2}$, so $||q_{2i}\frac{a}{p}||=\widehat{q_{2i}\frac{a}{p}}$ and $||q_{2i+2}\frac{a}{p}||=\widehat{q_{2i+2}\frac{a}{p}}$. By \ref{pro:contfrac1}, $\widehat{q_{2i}\frac{a}{p}}>\widehat{q_{2i+2}\frac{a}{p}}$.

As a previous step, we show that $M_{q_{2i,r}}$ lies between $M_{q_{2i+2}}$ and $M_{q_{2i}}$ for every $0\leq r\leq a_{2i+2}$. Since $q_{2i}<q_{2i,r}$, \ref{pro:contfrac2} yields $\widehat{q_{2i,r}\frac{a}{p}}<\widehat{q_{2i}\frac{a}{p}}$. Likewise, since $q_{2i,r}<q_{2i+2}$, $||q_{2i,r}\frac{a}{p}||>\widehat{q_{2i+2}\frac{a}{p}}$. But from the already proved inequality we deduce that $\widehat{q_{2i,r}\frac{a}{p}}<\frac{1}{2}$, so $\widehat{q_{2i,r}\frac{a}{p}}=||q_{2i,r}\frac{a}{p}||$.

Now, let $1\leq r\leq a_{2i+2}$. We have $$q_{2i,r}-q_{2i,r-1}=q_{2i+1}\quad\Longrightarrow\quad\Big|\Big|(q_{2i,r}-q_{2i,r-1})\frac{a}{p}\Big|\Big|=\Big|\Big|q_{2i+1}\frac{a}{p}\Big|\Big|,$$ proving that $d_C(q_{2i,r},q_{2i,r-1})=||q_{2i+1}\frac{a}{p}||$.

It remains to prove that $\widehat{q_{2i,r}\frac{a}{p}}<\widehat{q_{2i,r-1}\frac{a}{p}}$. We have $$\widehat{q_{2i,r}\frac{a}{p}}=\widehat{(q_{2i+1}+q_{2i,r-1})\frac{a}{p}}.$$ If we prove that $\widehat{q_{2i+1}\frac{a}{p}}+\widehat{q_{2i,r-1}\frac{a}{p}}>1$, then we obtain that $\widehat{q_{2i,r}\frac{a}{p}}=\widehat{q_{2i+1}\frac{a}{p}}+\widehat{q_{2i,r-1}\frac{a}{p}}-1<\widehat{q_{2i,r-1}\frac{a}{p}}$, obtaining the desired inequality. Now, the inequality $\widehat{q_{2i+1}\frac{a}{p}}+\widehat{q_{2i,r-1}\frac{a}{p}}>1$ is equivalent to $\widehat{q_{2i,r-1}\frac{a}{p}}>1-\widehat{q_{2i+1}\frac{a}{p}}=||q_{2i+1}\frac{a}{p}||$, the last equality due to \ref{pro:contfrac3}. Since $M_{q_{2i,r-1}}$ lies between $M_{q_{2i+2}}$ and $M_{q_{2i}}$, $\widehat{q_{2i,r-1}\frac{a}{p}}=d_C(q_{2i,r-1},0)>d_C(q_{2i,r-1},q_{2i,r})=||q_{2i+1}\frac{a}{p}||$, as we wanted. 
\end{proof}

With this, we can prove the following technical but useful result:

\begin{lema}\label{lem:density} Let $h\in\mathbb{Z}$, $1\leq h<p$. If $M_h$ lies between $M_{q_{2i+2}}$ and $M_{q_{2i}}$ for some $0\leq i<\frac{n-1}{2}$, there is $0\leq r\leq a_{2i+2}$ such that $d_C(h,q_{2i,r})<||q_{2i+1}\frac{a}{p}||$. If in addition $|h-q_{2i,r}|<q_{2i+2}$, then $h=q_{2i,r}$.
\end{lema}
\begin{proof}
By Proposition \ref{pro:arcsemiconv}, there is some $0\leq r<a_{2i+2}$ such that $M_h$ lies between $M_{q_{2i,r}}$ and $M_{q_{2i,r+1}}$. We can assume without loss of generality that $h\neq q_{2i,r+1}$; otherwise we would replace $r$ by $r+1$. Then $d_C(h,q_{2i,r})<d_C(q_{2i,r},q_{2i,r+1})=||q_{2i+1}\frac{a}{p}||$. Suppose in addition that $|h-q_{2i,r}|<q_{2i+2}$. If $h\neq q_{2i,r}$, we have from Proposition \ref{pro:contfrac} \ref{pro:contfrac2} that $d_C(h,q_{2i,r})\geq||q_{2i+1}\frac{a}{p}||$, which is a contradiction. Then $h=q_{2i,r}$.
\end{proof}

\begin{coro}\label{coro:density} Let $h\in\mathbb{Z}$, $1\leq h<p$. 
\begin{enumerate}[label=\arabic*)]
    \item\label{coro:density1} If $M_h$ lies between $M_{q_{2i+2}}$ and $M_{q_{2i}}$ and $h<q_{2i+2}$, then $h=q_{2i,r}$ for some $0\leq r\leq a_{2i+2}$.
    \item\label{coro:density2} If $M_h$ lies between $M_{q_{2i+2}}$ and $M_{q_{2i,r}}$ for some $0\leq r\leq a_{2i+2}$ and $h<q_{2i+2}$, then $h=q_{2i,r'}$ for some $r\leq r'\leq a_{2i+2}$.
\end{enumerate}
\end{coro}
\begin{proof}
Note that \ref{coro:density1} is a particular case of \ref{coro:density2}, so it is enough to prove the latter. Suppose that $M_h$ lies between $M_{q_{2i+2}}$ and $M_{q_{2i,r}}$ for some $0\leq r\leq a_{2i+2}$ and $h<q_{2i+2}$. By Proposition \ref{pro:arcsemiconv}, $M_h$ lies between $M_{q_{2i+2}}$ and $M_{q_{2i}}$. Using Lemma \ref{lem:density}, we have that $d_C(h,q_{2i,r'})<||q_{2i+1}\frac{a}{p}||$ for some $0\leq r'\leq a_{2i+2}$. Now, since $h<q_{2i+2}$, we have $|h-q_{2i,r'}|<q_{2i+2}$, so $h=q_{2i,r'}$. Finally, our hypothesis becomes $\widehat{q_{2i,r'}\frac{a}{p}}>\widehat{q_{2i,r}\frac{a}{p}}$, so from Proposition \ref{pro:arcsemiconv} we see that $r\leq r'$.
\end{proof}

This setup allows us to rebuild the proof of Proposition \ref{pro:paramE}.

\begin{proof} \textit{(of Proposition \ref{pro:paramE})}
We consider the results at Proposition \ref{pro:contfrac}. From \ref{pro:contfrac2} and \ref{pro:contfrac3} we see immediately that $q_{2i}\in E$ for all $0\leq i<\frac{n}{2}$.

Let us prove that all the numbers $q_{2i,r}$ as in the statement belong to $E$. Fix $1\leq r\leq a_{2i+2}$ and call $h=q_{2i,r}$. Pick $1\leq h'<h$ and assume that $\widehat{h'\frac{a}{p}}<\widehat{h\frac{a}{p}}$. Since $h'<q_{2i+2}$ and $q_{2i+2}\in E$, we have that $\widehat{q_{2i+2}\frac{a}{p}}<\widehat{h'\frac{a}{p}}$, so $M_{h'}$ is between $M_{q_{2i+2}}$ and $M_h$. By Corollary \ref{coro:density}, $h'=q_{2i,r'}$ for some $r\leq r'\leq a_{2i+2}$. But $r'\geq r$ implies $h'\geq h$, which is a contradiction. Hence $\widehat{h'\frac{a}{p}}\geq\widehat{h\frac{a}{p}}$, and since $1\leq h'<h<p$, the equality is not possible.

Finally, we prove that there are no more integer numbers in $E$. Let $h\in E$. If $h$ is an even-indexed convergent, we have finished. Otherwise, we know that the sequence of even-indexed convergents is strictly increasing, so $h$ lies necessarily between two of them. In other words, there is $0<i<\frac{n-1}{2}$ such that $q_{2i}<h<q_{2i+2}$. Since these integers belong to $E$, we have that $M_h$ lies between $M_{q_{2i+2}}$ and $M_{q_{2i}}$. By Corollary \ref{coro:density}, $h=q_{2i,r}$ for some $0\leq r\leq a_{2i+2}$.
\end{proof}

\subsection{The matrices $M(\pi_K^{-d(k)}\Psi^k)$}\label{sectmatrices}

In this section we determine the values of the entries $\overline{\mu}_{j,i}^{(k)}$ following the expression \eqref{eq:coordbasis}. First, we rewrite the numbers $w(i)$ in terms of the set $E$.

The situation changes substantially depending on whether or not $k+i\leq p-1$, so we distinguish these two cases.

\begin{pro}\label{pro:entriesleqp-1} Fix $0\leq k,i\leq p-1$ such that $k+i\leq p-1$, and call $h=p-i$.
\begin{enumerate}[label=\arabic*)]
    \item\label{pro:entriesleqp-11} If $j\neq k+i$, then $\overline{\mu}_{j,i}^{(k)}=0$.
    \item\label{pro:entriesleqp-12} If $h\notin E$, $$\overline{\mu}_{k+i,i}^{(k)}=\begin{cases}
    1 & \hbox{if }h=p\hbox{ or }\widehat{(k+1)\frac{a}{p}}<\widehat{h\frac{a}{p}}\\
    0 & \hbox{otherwise}\end{cases}$$
    \item\label{pro:entriesleqp-13} If $h\in E$, $\overline{\mu}_{k+i,i}^{(k)}=1$.
\end{enumerate}
\end{pro}
\begin{proof}
This is essentially the proof of \cite[Lemme 2.a]{bertrandiasbertrandiasferton} and \cite[Proposition 7.10]{gildegpdihedral}.

Since $k+i\leq p-1$, $\Psi_{k+i}$ is already written with respect to the $K$-basis $\{\Psi^l\}_{l=0}^{p-1}$. Then \ref{pro:entriesleqp-11} is trivial.

Let us prove \ref{pro:entriesleqp-12} and \ref{pro:entriesleqp-13}. Since $C(\Psi^{k+i},\Psi^{k+i})=1$, $\mu_{k+i,i}^{(k)}=\pi_K^{d(k+i)-d(k)-w(i)}$. By Lemma \ref{lem:nuiE}, $$d(k)+w(i)=(k+i)a_0+\Big\lfloor i\frac{a}{p}\Big\rfloor+\Big\lfloor(k+1)\frac{a}{p}\Big\rfloor+\epsilon,$$ where $\epsilon=1$ if $h\in E$ and $\epsilon=0$ otherwise. Now, we have that $$\Big\lfloor(i+k+1)\frac{a}{p}\Big\rfloor-\Big\lfloor i\frac{a}{p}\Big\rfloor-\Big\lfloor(k+1)\frac{a}{p}\Big\rfloor=\widehat{i\frac{a}{p}}+\widehat{(k+1)\frac{a}{p}}-\widehat{(i+k+1)\frac{a}{p}},$$ so $$d(k+i)-d(k)-w(i)=\widehat{i\frac{a}{p}}+\widehat{(k+1)\frac{a}{p}}-\widehat{(i+k+1)\frac{a}{p}}-\epsilon.$$ 

Assume that $h\notin E$, so that $$d(k+i)-d(k)-w(i)=\widehat{i\frac{a}{p}}+\widehat{(k+1)\frac{a}{p}}-\widehat{(i+k+1)\frac{a}{p}}.$$ If $h=p$, then $i=0$ and trivially $d(k+i)=d(k)+w(i)$, so $\overline{\mu}_{k+i,i}^{(k)}=1$. Otherwise, we have the equality $d(k+i)=d(k)+w(i)$ if and only if $\widehat{(k+1)\frac{a}{p}}+\widehat{i\frac{a}{p}}<1$, that is, $\widehat{(k+1)\frac{a}{p}}<\widehat{h\frac{a}{p}}$.

Assume that $h\in E$. In this case, we have $$d(k+i)-d(k)-w(i)=\widehat{i\frac{a}{p}}+\widehat{(k+1)\frac{a}{p}}-\widehat{(i+k+1)\frac{a}{p}}-1.$$ It is clear that $\widehat{(k+1)\frac{a}{p}}+\widehat{i\frac{a}{p}}<2$. Then $d(k+i)-d(k)-w(i)\leq0$, and it is also non-negative because $\mu_{k+i,i}^{(k)}\in\mathcal{O}_K$. Hence $d(k+i)-d(k)-w(i)=0$ and $\overline{\mu}_{k+i,i}^{(k)}=1$.
\end{proof}


Let us move on to the case $k+i\geq p$. In this case, in order to determine the $K$-valuations of the coefficients $C(\Psi^j,\Psi^{k+i})$, we rewrite $\Psi^{k+i}$ using the equality $\Psi^p=\lambda\Psi$ from Corollary \ref{coro:eqdegp}. Again, due to the simplicity of this equation, for the cases $r=1$ and $r=2$, we obtain a much simpler description than the ones at \cite[Lemme 2]{bertrandiasbertrandiasferton} and \cite[Proposition 7.11]{gildegpdihedral}, respectively.

\begin{pro}\label{pro:entriesgeqp} Fix $0\leq k,i\leq p-1$ such that $k+i\geq p$, call $h=p-i$ and let $m=k+i-(p-1)$.
\begin{enumerate}[label=\arabic*)]
    \item\label{pro:entriesgeqp1} If $j\neq m$, then $\overline{\mu}_{j,i}^{(k)}=0$.
    \item\label{pro:entriesgeqp2} If $h\notin E$, $\overline{\mu}_{m,i}^{(k)}\neq0$ if $\frac{a}{p}+\widehat{(k+1)\frac{a}{p}}<\widehat{h\frac{a}{p}}$, and $\overline{\mu}_{m,i}^{(k)}=0$ otherwise.
    \item\label{pro:entriesgeqp3} If $h\in E$, $\overline{\mu}_{m,i}^{(k)}\neq0$ if $\frac{a}{p}+\widehat{(k+1)\frac{a}{p}}<\widehat{h\frac{a}{p}}+1$, and $\overline{\mu}_{m,i}^{(k)}=0$ otherwise.
\end{enumerate}
\end{pro}
\begin{proof}
We adapt the arguments at \cite[Lemme 2.b]{bertrandiasbertrandiasferton} and \cite[Proposition 7.11]{gildegpdihedral} to our case.

From Corollary \ref{coro:eqdegp}, we have that $$\Psi^{k+i}=\Psi^{m-1}\Psi^p=\lambda\Psi^m.$$ Hence, $$C(w^j,w^{k+i})=\begin{cases}
    \lambda & \hbox{if }j=m, \\
    0 & \hbox{otherwise}.
\end{cases}$$ Looking at \eqref{eq:coordbasis}, \ref{pro:entriesgeqp1} follows immediately. On the other hand, \begin{equation*}
    \begin{split}
        v_K(\mu_{m,i}^{(k)})&=e-\frac{p-1}{r}c+d(m)-d(k)-w(i)\\&=e-\frac{p-1}{r}c+ma_0+\Big\lfloor(m+1)\frac{a}{p}\Big\rfloor-(k+i)a_0-\Big\lfloor i\frac{a}{p}\Big\rfloor-\Big\lfloor(k+1)\frac{a}{p}\Big\rfloor-\varepsilon,
    \end{split}
\end{equation*} where $\varepsilon=1$ if $h\in E$ and $\varepsilon=0$ otherwise. Since $ma_0=(k+i)a_0-(p-1)a_0$ and $$\Big\lfloor(m+1)\frac{a}{p}\Big\rfloor-\Big\lfloor i\frac{a}{p}\Big\rfloor-\Big\lfloor(k+1)\frac{a}{p}\Big\rfloor=-(p-1)\frac{a}{p}-\widehat{(m+1)\frac{a}{p}}+\widehat{i\frac{a}{p}}+\widehat{(k+1)\frac{a}{p}},$$ this can be rewritten as $$v_K(\mu_{m,i}^{(k)})=e-\frac{p-1}{r}c-(p-1)a_0-(p-1)\frac{a}{p}-\widehat{(m+1)\frac{a}{p}}+\widehat{i\frac{a}{p}}+\widehat{(k+1)\frac{a}{p}}-\epsilon.$$ Now, note that since $\ell\geq\frac{pe}{p-1}-1$, from Proposition \ref{pro:condscaffoldp} it follows that $\mathfrak{c}=a$, and hence by Proposition \ref{pro:precscaffold}, $e-\frac{p-1}{r}c=d(p-1)=a+(p-1)a_0$. Therefore, 
$$v_K(\mu_{m,i}^{(k)})=\frac{a}{p}-\widehat{(m+1)\frac{a}{p}}+\widehat{i\frac{a}{p}}+\widehat{(k+1)\frac{a}{p}}-\epsilon.$$ Replacing $i=p-h$, we have $$v_K(\mu_{m,i}^{(k)})=\frac{a}{p}-\widehat{(m+1)\frac{a}{p}}+1-\widehat{h\frac{a}{p}}+\widehat{(k+1)\frac{a}{p}}-\epsilon.$$ Let $\gamma=\frac{a}{p}-\widehat{h\frac{a}{p}}+\widehat{(k+1)\frac{a}{p}}$, so that $$v_K(\mu_{m,i}^{(k)})=\gamma-\widehat{(m+1)\frac{a}{p}}+1-\epsilon.$$ Since $\frac{a}{p}-h\frac{a}{p}+(k+1)\frac{a}{p}=(m+1)\frac{a}{p}$, we have $$\gamma=\begin{cases}
\widehat{(m+1)\frac{a}{p}} & \hbox{if }0\leq\gamma<1, \\
\widehat{(m+1)\frac{a}{p}}-1 & \hbox{if }-1\leq\gamma<0, \\
\widehat{(m+1)\frac{a}{p}}+1 & \hbox{if }1\leq\gamma<2, \\
\end{cases}$$ Therefore, $$v_K(\mu_{j,i}^{(k)})=\begin{cases}
1-\epsilon & \hbox{if }0\leq\gamma<1, \\
-\epsilon & \hbox{if }-1\leq\gamma<0, \\
2-\epsilon & \hbox{if }1\leq\gamma<2. \\
\end{cases}$$

Assume that $h\notin E$, so $\epsilon=0$. Then $v_K(\mu_{j,i}^{(k)})=0$ if and only if $-1\leq\gamma<0$. The first inequality trivially holds, while the second one is equivalent to $\frac{a}{p}+\widehat{(k+1)\frac{a}{p}}<\widehat{h\frac{a}{p}}$. This proves \ref{pro:entriesgeqp2}. Otherwise, if $h\in E$, then $\epsilon=1$ and $v_K(\mu_{j,i}^{(k)})=0$ if and only if $0\leq\gamma<1$. If $\widehat{h\frac{a}{p}}>\widehat{(k+1)\frac{a}{p}}$, since $h\in E$ we have by definition that $h\leq k+1$, which contradicts our hypothesis, so $\widehat{h\frac{a}{p}}\leq\widehat{(k+1)\frac{a}{p}}$. Therefore, the first inequality in the chain above always holds. On the other hand, the second inequality may be rewritten as $\frac{a}{p}+\widehat{(k+1)\frac{a}{p}}<\widehat{h\frac{a}{p}}+1$, finishing the proof of \ref{pro:entriesgeqp3}.
\end{proof}

\begin{rmk}\normalfont The proof of Proposition \ref{pro:entriesgeqp} shows that the presence of $v_K(\lambda)=e-\frac{p-1}{r}c$ in the expression of the precision $\mathfrak{c}$ of the scaffold is more than just a coincidence; it is used to equate it with $a+(p-1)a_0$, which is a fundamental step to finish the proof.
\end{rmk}

\subsection{The $\mathfrak{A}_{L/K}$-freeness of $\mathcal{O}_L$}

In this section we carry out the proof of Theorem \ref{maintheorem} \ref{mainthm4}.

\begin{pro}\label{pro:sufficiency} If $n\leq 4$, then $\mathcal{O}_L$ is $\mathfrak{A}_{L/K}$-free.
\end{pro}
\begin{proof}
We adapt the proof at \cite[Proposition 2]{bertrandiasbertrandiasferton} to this situation.

We know from Corollary \ref{coro:scaffoldp} \ref{coro:scaffoldp2} and Proposition \ref{pro:charactatheta} that $\mathcal{O}_L$ is $\mathfrak{A}_{L/K}$-free under the assumption that $n\leq2$. Let us assume that $n\in\{3,4\}$.

Call $k=q_2-1$. For each $u\in\mathcal{O}_K^{\times}$, let $$\alpha_u=u\mathrm{Id}_H+\pi_K^{-d(k)}\Psi^k\in\mathfrak{A}_{\theta},$$ so that $M(\alpha_u)=uM(1)+M(\pi_K^{-d(k)}\Psi^k)$. Clearly, the matrix $M(1)$ is diagonal and from \eqref{eq:coordbasis} we see that $\overline{\mu}_{ii}^{(0)}=1$ (resp. $0$) if $w(i)=d(i)$ (resp. $w(i)<d(i)$). Since $n>2$, we have that $\mathfrak{A}_{L/K}\neq\mathfrak{A}_{\theta}$, so there is some $0$ in the diagonal of $M(1)$. We proceed to study the matrix $M(\pi_K^{-d(k)}\Psi^k)$ for $k>0$. Throughout this proof, we fix $0\leq i\leq p-1$ and call $h=p-i$.

First, we study the entries $\overline{\mu}_{j,i}^{(k)}$ with $k+i\leq p-1$. If $j\neq k+i$, we have directly that $\overline{\mu}_{j,i}^{(k)}=0$ by Proposition \ref{pro:entriesleqp-1} \ref{pro:entriesleqp-11}. As for $j=k+i$, we claim that $\overline{\mu}_{k+i,i}^{(k)}=1$. By Proposition \ref{pro:entriesleqp-1} \ref{pro:entriesleqp-12}, it is enough to prove that if $h\notin E$ and $h\neq p$, then $\widehat{(k+1)\frac{a}{p}}<\widehat{h\frac{a}{p}}$. Suppose that $n=3$. By Proposition \ref{pro:contfrac} \ref{pro:contfrac3}, \ref{pro:contfrac4}, $\widehat{(k+1)\frac{a}{p}}=\widehat{q_2\frac{a}{p}}=\frac{1}{p}$, so $\widehat{(k+1)\frac{a}{p}}\leq\widehat{h\frac{a}{p}}$. Since $k+1<h<p$, the inequality is strict. Assume now that $n=4$, and suppose that $\widehat{h\frac{a}{p}}<\widehat{(k+1)\frac{a}{p}}=\widehat{q_2\frac{a}{p}}$. Since $q_4=p$, $M_h$ lies between $M_{q_4}$ and $M_{q_2}$. Now, we have that $h<q_4$ because $h\neq p$, so $h=q_{2,r}$ for some $0\leq r<a_4$ by Corollary \ref{coro:density}. Hence $h\in E$, which is a contradiction.

Next, we assume that $k+i\geq p$ and consider the entries $\overline{\mu}_{j,i}^{(k)}$. Let us call $m_i=k+i-(p-1)$. If $j\neq m_i$, we have by Proposition \ref{pro:entriesgeqp} \ref{pro:entriesgeqp1} that $\overline{\mu}_{j,i}^{(k)}=0$. We claim that $\overline{\mu}_{m_i,i}^{(k)}\neq0$. 

Assume that $h\in E$. Note that the inequality $k+i\leq p$ is equivalent to $h<k+1=q_2$. Since $q_2\in E$, $\widehat{h\frac{a}{p}}>\widehat{q_2\frac{a}{p}}=\widehat{(k+1)\frac{a}{p}}$, so $\frac{a}{p}+\widehat{(k+1)\frac{a}{p}}<\widehat{h\frac{a}{p}}+1$ and we are done by Proposition \ref{pro:entriesgeqp} \ref{pro:entriesgeqp3}. On the contrary, suppose that $h\notin E$. By Proposition \ref{pro:entriesgeqp} \ref{pro:entriesgeqp2}, it is enough to prove that $\frac{a}{p}+\widehat{(k+1)\frac{a}{p}}<\widehat{h\frac{a}{p}}$. Suppose that $\widehat{h\frac{a}{p}}<\frac{a}{p}$. If $n=3$, then $\widehat{q_2\frac{a}{p}}=\frac{1}{p}$, so \begin{equation}\label{eq:chain1}
    \widehat{q_2\frac{a}{p}}<\widehat{h\frac{a}{p}}<\frac{a}{p}.
\end{equation} If $n=4$, we have either this chain of inequalities or alternatively \begin{equation}\label{eq:chain2}
    \widehat{q_4\frac{a}{p}}<\widehat{h\frac{a}{p}}<\widehat{q_2\frac{a}{p}}.
\end{equation} Note that we have $h<q_2<q_4$. For $n\in\{3,4\}$, if \eqref{eq:chain1} holds, then Corollary \ref{coro:density} gives that $h=q_{2,r}$ for some $0\leq r\leq a_2$. Otherwise, if $n=4$ and \eqref{eq:chain2} is satisfied, then $h=q_{2,r}$ for some $0\leq r\leq a_4$ by Corollary \ref{coro:density}. In both cases, we have that $h\in E$, a contradiction. Hence, $\widehat{h\frac{a}{p}}>\frac{a}{p}$. Note that $h\neq 1$ because $1=q_0\in E$ and $h\notin E$. Then, $1\leq h-1<q_2$ with $q_2\in E$, so $\widehat{(h-1)\frac{a}{p}}>\widehat{q_2\frac{a}{p}}$. Since $\widehat{h\frac{a}{p}}>\frac{a}{p}$, necessarily $\widehat{(h-1)\frac{a}{p}}=\widehat{h\frac{a}{p}}-\frac{a}{p}$. Therefore, $\widehat{h\frac{a}{p}}>\frac{a}{p}+\widehat{(k+1)\frac{a}{p}}$, and the claim follows from Proposition \ref{pro:entriesgeqp} \ref{pro:entriesgeqp3}.

Therefore, the first row of $\overline{M(\alpha)}$ is $(\overline{u},0,\dots,0)$, and developing the determinant by this row, we obtain $\mathrm{det}(\overline{M(\alpha)})=\overline{u}\mathrm{det}(M')$, where $M'\in\mathcal{M}_{p-1}(\mathcal{O}_K/\mathfrak{p}_K)$. We number the entries of $M'$ with subscripts starting in $1$. Since $\overline{\mu}_{0,0}^{(0)}=1$ and the diagonal of $M(1)$ has some zero, so does the diagonal of $M'$, and the non-zero diagonal elements in $M'$ are $\overline{u}$. 

The entries $\overline{\mu}_{j,i}^{(k)}$ at the first $p-k-1$ columns of $M'$ have $k+i\leq p-1$, so the terms above the main diagonal of $M'$ are zero, and the only non-zero terms below are the entries of a subdiagonal of $1$'s in positions $(k+i,i)$, $1\leq i\leq p-k-1$. As for the entries at the other columns, for which $k+i\geq p$, let $m_i=k+i-(p-1)$. We already know that $\overline{\mu}_{m_i,i}^{(k)}\neq0$ and $\overline{\mu}_{j,i}^{(k)}=0$ for $j\neq m_i$. The subscripts of all the non-zero terms correspond to (the inverse of) the permutation of $(1,\dots,p-1)$ that sends $i$ to $k+i$ if $1\leq i\leq p-k-1$ and to $m_i$ otherwise, so their product $$b_0=\prod_{i=1}^{p-k-1}\overline{\mu}_{k+i,i}^{(k)}\prod_{i=p-k}^{p-1}\overline{\mu}_{m_i,i}^{(k)}$$ is the unique non-zero summand of $\mathrm{det}(M')$ independent of $\overline{u}$.

Since the diagonal of $M'$ has at least one zero entry (and hence independent of $\overline{u}$), we deduce that there is a polynomial $P\in\mathcal{O}_K[X]$ of degree at most $p-2$ and independent of $u$ such that $\mathrm{det}(M(\alpha))\equiv uP(u)\,(\mathrm{mod}\,\mathfrak{p}_K)$ and $\overline{P(0)}\neq0$. Since $P$ is of degree at most $p-2$, we can choose $u$ such that $\overline{P(u)}\neq0$, and therefore $\mathrm{det}(M(\alpha))\not\equiv0\,(\mathrm{mod}\,\mathfrak{p}_K)$.

We have proved that for $\theta=\pi_L^a$, $\mathfrak{A}_{\theta}$ is $\mathfrak{A}_{L/K}$-principal. By Proposition \ref{pro:hnormalfreeness}, $\mathcal{O}_L$ is $\mathfrak{A}_{L/K}$-free.
\end{proof}

It remains to prove the converse, for which we essentially adapt \cite[Proposition 3]{bertrandiasbertrandiasferton} to this situation. We shall assume that $n\geq 5$ and prove that $\mathfrak{A}_{\theta}$ is not $\mathfrak{A}_{L/K}$-principal. 

Let us call $n=2s+2$ if $n$ is even and $n=2s+1$ if $n$ is odd. We consider the integers $0\leq i\leq p-1$ such that $h=p-i$ takes one of the following values \begin{equation}
    \begin{split}\label{eq:valuesh}
        &2q_{2s-2}, \\
        &q_{2s-2,r}+q_{2s},\quad 0\leq r\leq a_{2s}.
    \end{split}
\end{equation}

Note that these values do not match with any of the elements at the description of $E$ from Proposition \ref{pro:paramE}, so they do not belong to $E$.

We make use of the following lemmas from \cite{bertrandiasbertrandiasferton,gildegpdihedral}.

\begin{lema}[\cite{gildegpdihedral}, Lemma 7.20]\label{lem:nonfreeleq} Suppose that $n\leq 5$. Given $0\leq i\leq p-1$, let $h=p-i$, and let $1\leq k\leq h-1$. If $0\leq j\leq p-1$ is such that $\overline{\mu}_{j,i}^{(k)}\neq0$, then $j=k+i$. Moreover:
\begin{enumerate}
    \item If $h=2q_{2s-2}$, then $k+1=q_{2s-2}$.
    \item If $h=q_{2s-2,r}+q_{2s}$ for some $0\leq r\leq a_{2s}$, then $k+1=q_{2s-2,r'}$ for some $r\leq r'\leq a_{2s}$.
\end{enumerate}
\end{lema}
\begin{proof}
The fact that $\overline{\mu}_{j,i}^{(k)}\neq0$ implies $j=k+i$ follows immediately from Proposition \ref{pro:entriesleqp-1}.

\begin{enumerate}
    \item Note that $$q_{2s}=a_{2s}q_{2s-1}+q_{2s-2}>2q_{2s-2}=h>k+1.$$ Since $q_{2s}\in E$, we deduce that $$\widehat{q_{2s}\frac{a}{p}}<\widehat{(k+1)\frac{a}{p}}.$$ Suppose that $\widehat{(k+1)\frac{a}{p}}>\widehat{q_{2s-2}\frac{a}{p}}$, so in particular $k+1\neq q_{2s-2}$. Then $$1\leq|k+1-q_{2s-2}|<q_{2s-2},$$ so $$d_C(k+1,q_{2s-2})>\widehat{q_{2s-2}\frac{a}{p}}.$$ But $$\widehat{k+1\frac{a}{p}}<\widehat{h\frac{a}{p}}\leq2\widehat{q_{2s-2}\frac{a}{p}},$$ and since $\widehat{q_{2s-2}\frac{a}{p}}<\widehat{(k+1)\frac{a}{p}}$, we obtain that $$d_C(k+1,q_{2s-2})=\widehat{(k+1)\frac{a}{p}}-\widehat{q_{2s-2}\frac{a}{p}}<\widehat{q_{2s-2}\frac{a}{p}},$$ which is a contradiction. We deduce that $$\widehat{(k+1)\frac{a}{p}}\leq\widehat{q_{2s-2}\frac{a}{p}}.$$ Hence, $M_{k+1}$ lies between $M_{q_{2s}}$ and $M_{q_{2s-2}}$. Since $k+1<q_{2s}$, Corollary \ref{coro:density} gives that $k+1=q_{2s-2,r}$ for some $0\leq r<a_{2s}$. But this is lower than $2q_{2s-2}$, so $r=0$ and $k+1=q_{2s-2}$.
    \item We claim that $$\widehat{(k+1)\frac{a}{p}}\geq\widehat{q_{2s}\frac{a}{p}}.$$ If $n=2s+1$, by Proposition \ref{pro:contfrac} \ref{pro:contfrac4} we have that $$\widehat{q_{2s}\frac{a}{p}}=\frac{1}{p}\leq\widehat{(k+1)\frac{a}{p}}.$$ Suppose now that $n=2s+2$. If $\widehat{(k+1)\frac{a}{p}}<\widehat{q_{2s}\frac{a}{p}}$, since $$k+1<h=q_{2s-2,r}+q_{2s}\leq 2q_{2s}<q_{2s+2},$$ we have by Corollary \ref{coro:density} that $k+1=q_{2s,r''}$ for some $0\leq r''\leq a_{2s+2}$. But then $k+1>2q_{2s}\geq h$, which is a contradiction. Then the claim follows. Next, we prove that $k+1\leq q_{2s}$. We have \begin{equation}\label{eq:ineqsemiconv}
        \widehat{(k+1)\frac{a}{p}}<\widehat{h\frac{a}{p}}\leq\widehat{q_{2s}\frac{a}{p}}+\widehat{q_{2s-2,r}\frac{a}{p}},
    \end{equation} whence $d_C(k+1,q_{2s})<\widehat{q_{2s-2,r}\frac{a}{p}}$. If $k+1>q_{2s}$, since $k+1<h=q_{2s}+q_{2s-2,r}$, we see that $$1\leq k+1-q_{2s}<q_{2s-2,r}$$ with $q_{2s-2,r}\in E$, so $$d_C(k+1,q_{2s})>\widehat{q_{2s-2,r}\frac{a}{p}},$$ which is a contradiction. Finally, we prove that $\widehat{(k+1)\frac{a}{p}}\leq\widehat{q_{2s-2,r}\frac{a}{p}}$. Assume, on the contrary, that $\widehat{(k+1)\frac{a}{p}}>\widehat{q_{2s-2,r}\frac{a}{p}}$. Since $1\leq k+1,q_{2s-2,r}\leq q_{2s}$, we have that $$d_C(k+1,q_{2s-2,r})>\widehat{q_{2s}\frac{a}{p}}.$$ But from \eqref{eq:ineqsemiconv} we obtain that $$d_C(k+1,q_{2s-2,r})<\widehat{q_{2s}\frac{a}{p}},$$ which is, once again, a contradiction. We have proved that $M_{k+1}$ is between $M_{q_{2s}}$ and $M_{q_{2s-2,r}}$ with $k+1\leq q_{2s}$. Thus, by Corollary \ref{coro:density}, $k+1=q_{2s-2,r'}$ for some $r\leq r'\leq a_{2s}$.
\end{enumerate}
\end{proof}

\begin{lema}[\cite{gildegpdihedral}, Lemma 7.21]\label{lem:nonfreegeq} Suppose that $n\geq 5$. Let $h$ be as in \eqref{eq:valuesh} and let $0\leq k\leq p-1$. If $k+1>h$ or $k=0$, then $\mu_{j,i}^{(k)}=0$ for every $0\leq j\leq p-1$.
\end{lema}
\begin{proof}
First, we prove that $\widehat{h\frac{a}{p}}<\frac{a}{p}$. Note that $\frac{a}{p}=d_C(1,0)$, the length of the arc in the circle $C$ to the point $M_1$ in anticlockwise direction.

If $h=2q_{2s-2}$, we have by Proposition \ref{pro:arcsemiconv} \begin{equation*}
    \begin{split}
        \frac{a}{p}=\widehat{q_{2s-2}\frac{a}{p}}+\sum_{j=1}^{s-1}a_{2j}\Big|\Big|q_{2j-1}\frac{a}{p}\Big|\Big|&\geq\widehat{q_{2s-2}\frac{a}{p}}+\Big|\Big|q_{2s-3}\frac{a}{p}\Big|\Big|\\&>2\widehat{q_{2s-2}\frac{a}{p}}=\widehat{h\frac{a}{p}}.
    \end{split}
\end{equation*} On the contrary, if $h=q_{2s-2,r}+q_{2s}$, then \begin{equation*}
    \begin{split}
        \frac{a}{p}=\widehat{q_{2s}\frac{a}{p}}+\sum_{j=1}^{s}a_{2j}\Big|\Big|q_{2j-1}\frac{a}{p}\Big|\Big|&\geq\widehat{q_{2s}\frac{a}{p}}+r\Big|\Big|q_{2s-1}\frac{a}{p}\Big|\Big|+\Big|\Big|q_{2s-3}\frac{a}{p}\Big|\Big|\\&>\widehat{q_{2s}\frac{a}{p}}+\widehat{q_{2s-2,r}\frac{a}{p}}\geq\widehat{h\frac{a}{p}}.
    \end{split}
\end{equation*}

For $k=0$, let us prove that $\overline{\mu}_{i,i}^{(0)}=0$, i.e, $d(i)\neq w(i)$. We have that $d(i)=ia_0+\Big\lfloor(i+1)\frac{a}{p}\Big\rfloor$ and, since $h\notin E$, $w(i)=ia_0+\Big\lfloor i\frac{a}{p}\Big\rfloor$. Now, the inequality $\widehat{h\frac{a}{p}}<\frac{a}{p}$ amounts to $\frac{a}{p}+\widehat{i\frac{a}{p}}>1$. Then $\Big\lfloor(i+1)\frac{a}{p}\Big\rfloor=\Big\lfloor i\frac{a}{p}\Big\rfloor+1$, whence $d(i)\neq w(i)$.

Assume that $k+1>h$ and call $m=k+1-h$. If $j\neq m$, we know from Proposition \ref{pro:entriesgeqp} \ref{pro:entriesgeqp1} that $\mu_{j,i}^{(k)}=0$. As for $j=m$, since $h\notin E$ and $\widehat{h\frac{a}{p}}<\frac{a}{p}$, Proposition \ref{pro:entriesgeqp} \ref{pro:entriesgeqp2} gives that $\mu_{j,i}^{(k)}=0$.
\end{proof}

\begin{pro}\label{pro:necessity} If $n\geq5$, then $\mathcal{O}_L$ is not $\mathfrak{A}_{L/K}$-free.
\end{pro}
\begin{proof}
We shall prove that given any element $\alpha\in\mathcal{O}_L$, $\mathrm{det}(M(\alpha))\equiv0\,(\mathrm{mod}\,\mathfrak{p}_K)$. As usual, we number the columns of $M(\alpha)$ from $0$ to $p-1$.

By Lemmas \ref{lem:nonfreeleq} and \ref{lem:nonfreegeq}, the matrix $\overline{M(\alpha)}$ has $a_{2s}+2$ columns (the ones corresponding to \eqref{eq:valuesh}) with at most $a_{2s}+1$ non-zero elements, which are the ones in Lemma \ref{lem:nonfreeleq}. These elements are in the rows of index $j=k+i$, where $k$ is as stated therein.

If $h=2q_{2s-2}$, then $p-j-1=q_{2s-2}$. Otherwise, if $h=q_{2s-2,r}+q_{2s}$ for some $0\leq r\leq a_{2s}$, $$p-j-1=q_{2s-2}+(a_{2s}-(r'-r))q_{2s-1},\quad0\leq r\leq r'\leq a_{2s}.$$ In total, these correspond to $a_{2s}+1$ different values of $j$. Then the $a_{2s}+1$ possibly non-zero elements in each of these columns are allocated in the same rows. Since each summand of $\mathrm{det}(\overline{M(\alpha)})$ includes exactly one entry of each column as a factor, it is necessarily zero. This proves that $\mathfrak{A}_{\theta}$ is not $\mathfrak{A}_{L/K}$-principal for $\theta=\pi_L^a$. By Proposition \ref{pro:hnormalfreeness}, $\mathcal{O}_L$ is not $\mathfrak{A}_{L/K}$-free.
\end{proof}

Putting together Propositions \ref{pro:sufficiency} and \ref{pro:necessity}, we conclude the proof of Theorem \ref{maintheorem} \ref{mainthm4}.

\begin{example}\normalfont Freeness does not necessarily hold for degree $p$ extensions of $p$-adic fields $L/K$ such that $K/\mathbb{Q}_p$ is unramified (unlike the affirmative result when $r=2$, see \cite[Proposition 8.1]{gildegpdihedral}). Let us see a counterexample. Choose $p=13$, $e=1$, $r=12$ and $t=5$, and let $L/K$ be a typical degree $13$ extension of $13$-adic fields arising from these parameters by using Theorem \ref{thm:eldertypicalp} \ref{thm:eldertypicalp2}. We see that $\ell=\frac{5}{12}$, and $c=5$, so $b=\ell-\frac{pc}{r}=\frac{5}{12}-\frac{65}{12}=-5$. On the other hand, $\frac{pe}{p-1}=\frac{13}{12}$, so $\ell>\frac{pe}{p-1}-1$. Now, $$\frac{b}{p}=\frac{-5}{13}=[-1;1,1,1,1,2],$$ so $n=5$. From Theorem \ref{maintheorem} \ref{mainthm4}, we obtain that $\mathcal{O}_L$ is not $\mathfrak{A}_{L/K}$-free. 
\end{example}

In practice, we can apply Theorem \ref{maintheorem} with the continued fraction expansion of $\frac{a}{p}$ for simplicity. This does not affect the criteria at Theorem \ref{maintheorem} because $a$ is the remainder of $b$ mod $p$. Then the continued fraction expansions of $\frac{a}{p}$ and $\frac{b}{p}$ are the same up to the integer part, so in particular they have the same length. Concretely, we can do as follows: the equality $\ell=b+\frac{pc}{r}$ gives $t=br+pc$, from which $b\equiv r^{-1}t\,(\mathrm{mod}\,p)$, where $r^{-1}$ denotes the modular inverse of $r$ mod $p$.

\begin{example}\normalfont It may happen that freeness over the associated order holds for typical degree $p$ extensions of $p$-adic fields with small values of $e$, but it does not for typical degree $p$ extensions of $p$-adic fields with the same values of $p$, $t$ and $r$ but larger values of $e$. For instance, let $L/K$ be a typical degree $p$ extension of $p$-adic fields arising from Theorem \ref{thm:eldertypicalp} with $p=5$ and $r=4$. The ramification jump of $L/K$ is $\ell=\frac{t}{r}$, where $t$ is the ramification jump of $\widetilde{L}/K$, and must satisfy $0<\ell<\frac{pe}{p-1}=\frac{5e}{4}$, that is, $1\leq t<5e$. Recall that $\gcd(c,r)=1$ and $t\equiv c\,(\mathrm{mod}\,r)$, so we have also $\gcd(t,r)=1$. Since $r=4$, $t$ must be odd in our case. Let us study the $\mathfrak{A}_{L/K}$-freeness of $\mathcal{O}_L$ for $e\in\{1,2\}$.

If $e=1$, we obtain that $t\in\{1,3\}$ and $\frac{pe}{p-1}-1=\frac{1}{4}$, so $\ell\geq\frac{pe}{p-1}-1$. Then we use Theorem \ref{maintheorem} \ref{mainthm4}.
\begin{itemize}
    \item If $t=1$, then $a=4$, so $\frac{a}{p}=\frac{4}{5}=[0;1,4]$.
    \item If $t=3$, then $a=1$, so $\frac{a}{p}=\frac{2}{5}=[0;2,2]$.
\end{itemize} In both cases, $n=2$ and $\mathcal{O}_L$ is $\mathfrak{A}_{L/K}$-free.

If $e=2$, we obtain that $t\in\{1,3,7,9\}$. Since $\frac{pe}{p-1}-1=\frac{3}{2}$, for $t\in\{1,3\}$ we have $\ell<\frac{pe}{p-1}-1$ and we use Theorem \ref{maintheorem} \ref{mainthm3} to conclude that $\mathcal{O}_L$ is $\mathfrak{A}_{L/K}$-free, while for $t\in\{7,9\}$ we have $\ell>\frac{pe}{p-1}-1$ and we use Theorem \ref{maintheorem} \ref{mainthm4}. Since the calculation of $a$ does not depend on $e$, for $t\in\{1,3\}$ we have that $n=2$ from the previous case, so $\mathcal{O}_L$ is $\mathfrak{A}_{L/K}$-free. As for the remaining cases:
\begin{itemize}
    \item If $t=7$, then $a=3$, so $\frac{a}{p}=\frac{3}{5}=[0;1,1,2]$. Since $n=3$, $\mathcal{O}_L$ is $\mathfrak{A}_{L/K}$-free.
    \item If $t=9$, then $a=1\mid p-1=13$, so $\mathcal{O}_L$ is $\mathfrak{A}_{L/K}$-free.
\end{itemize}

Hence $\mathcal{O}_L$ is $\mathfrak{A}_{L/K}$-free if $e\leq2$. However, this does not necessarily hold for $e\geq3$. For $t=7$ we have that $5\nmid t$, $(t,4)=1$ and $t<5e$, so we can take $L/K$ from these choices. Now, we have that $\frac{pe}{p-1}-1\geq\frac{11}{4}$, so the criterion we can apply is the one at Theorem \ref{maintheorem} \ref{mainthm3}. Since $n=3$, we conclude that $\mathcal{O}_L$ is not $\mathfrak{A}_{L/K}$-free.
\end{example}

\section*{Acknowledgements}

The author would like to thank the referee, whose remarks and suggestions led to an important improvement of this paper. The author is also grateful with Ilaria Del Corso and Griffith Elder for insightful discussions. This work was supported by the grant PID2022-136944NB-I00 (Ministerio de Ciencia e Innovación), and by Czech Science Foundation, grant 24-11088O.

\printbibliography


\end{document}